\documentclass[11pt]{amsart}
\usepackage{amsmath,amsfonts,amssymb,mathabx}
\usepackage{amsthm}
\usepackage{amstext}
\usepackage[normalem]{ulem}
\usepackage{verbatim}
\usepackage{longtable}
 
\usepackage{enumitem}
\usepackage{hyperref}
\hypersetup{
	colorlinks,
	linkcolor={blue},
	citecolor={red},
	urlcolor={blue}
}
 
\makeatletter
\newcommand{\setword}[2]{%
	\phantomsection
	#1\def\@currentlabel{\unexpanded{#1}}\label{#2}%
}
\makeatother
 
\usepackage[usenames,dvipsnames]{color}
\usepackage{imakeidx}
\scrollmode

\newtheorem{thm}{Theorem}[section]
\newtheorem{lem}[thm]{Lemma}
\newtheorem{cor}[thm]{Corollary}
\newtheorem{rem}[thm]{Remark}
\newtheorem{rems}[thm]{Remarks}

\newtheorem{clm}[thm]{Claim}
\newtheorem{prop}[thm]{Proposition}
\newtheorem{df}[thm]{Definition}

\newtheorem{ex}[thm]{Example}

\newtheorem{symb}[thm]{Notation}
\newtheorem{symbs}[thm]{Notations}

\newcommand{\wh}{\widehat}
\newcommand{\wt}{\widetilde}
\newcommand{\wc}{\widecheck}

\newcommand{\vS}{\varSigma}

\newcommand{\vO}{\varOmega}
\newcommand{\vT}{\varTheta}
\newcommand{\ov}{\overline}
\newcommand{\mf}{\mathfrak}
\newcommand{\vY}{\varUpsilon}

\newcommand{\uph}{\upharpoonright}

\newcommand{\R}{\mathbb{R}}
\newcommand{\T}{\mathbb{T}}
\newcommand{\N}{\mathbb{N}}
\newcommand{\F}{\mathcal{F}}

\newcommand{\D}{\mathcal{D}}
 
\newcommand{\B}{\mathfrak{B}} 
\newcommand{\M}{\mathcal{M}} 
\newcommand{\G}{\mathcal{G}} 
\newcommand{\C}{\mathcal{C}} 
\newcommand{\E}{\mathbb{E}}

\newcommand{\al}{\alpha}
 
\newcommand{\ga}{\gamma}
\newcommand{\be}{\beta}

\newcommand{\ttheta}{\rho(\theta)}
\newcommand{\Ttheta}{\rho(\vT)}

\begin{document}

 \title{A characterization of progressively equivalent probability measures preserving the structure of a    compound mixed renewal process}

 	   \author{S.~M. Tzaninis}
 	\address{Department of Statistics and Insurance Science\\ University of
 		Piraeus\\ 80 Karaoli and Dimitriou str.\\ 185 34 Piraeus\\ Greece}
 	\email{stzaninis@unipi.gr}
 	\thanks{}
 	
 	\author{N.~D. Macheras}
 	\address{Department of Statistics and Insurance Science\\ University of
 		Piraeus\\ 80 Karaoli and Dimitriou str.\\ 185 34 Piraeus\\ Greece} 
 	\email{macheras@unipi.gr} 
 	\thanks{}

\date{\today}

\begin{abstract}
Generalizing earlier works of Delbaen \& Haezendonck \cite{dh} as well as of \cite{mt3} and  \cite{lm3} for given compound mixed renewal process $S$ under a probability measure $P$, we characterize all those probability measures $Q$ on the domain of $P$ such that $Q$ and $P$ are  progressively equivalent  and $S$ remains a compound mixed renewal process under $Q$ with improved properties. As a consequence, we prove that any compound mixed renewal process can be converted into a compound mixed Poisson process through a change of measures. Applications related to the ruin problem and to the computation of premium calculation principles in an insurance market without arbitrage opportunities are discussed in \cite{t1} and \cite{mt6a}, respectively.
\smallskip

\smallskip

\noindent
{\bf AMS Subject Classification (2010):} \noindent Primary 60G55, 91B30 ; secondary 28A35, 60A10, 60G44, 60K05.
\smallskip
\end{abstract}
\maketitle
\section{Introduction}\label{intro}
A basic method in mathematical finance is to replace the
original probability measure with an equivalent martingale measure, sometimes called
a risk-neutral measure. This measure is used for pricing and hedging given contingent
claims (e.g. options, futures, etc.).   In contrast to the situation of the classical Black-Scholes option pricing formula, where the equivalent martingale measure is unique, in actuarial mathematics that is certainly not the case.

The above fact was pointed out by Delbaen \& Haezendonck in their pioneering paper \cite{dh}, as the authors ``tried to create a mathematical framework to deal with finance related to risk processes" in the frame of classical Risk Theory. Thus, they were confronted with the problem of characterizing all equivalent martingale measures $Q$ such that a compound Poisson process under an original measure $P$ remains a compound Poisson one under $Q$. Delbaen \& Haezendonck answered to the positive the previous problem in \cite{dh}, and applied their results to the theory of premium calculation principles (see also \cite{em} for an overview). The method provided by \cite{dh} has been successfully applied to many areas of insurance mathematics such as pricing (re-)insurance contracts (Holtan \cite{ho}, Haslip \&  Kaishev \cite{haka}), simulation of ruin probabilities (Boogaert \& De  Waegenaere \cite{bw}), risk capital allocation (Yu et al. \cite{yu}), pricing CAT derivatives (Geman \& Yor \cite{gy}, Embrechts \& Meister \cite{emme}). 

However, there is one vital point about the (compound) Poisson processes which is
their greatest weakness as far as practical applications are considered, and this is the fact that the variance is a linear function of time $t$. The latter consideration, together with  some interesting real-life cases of interest in Risk Theory, which show that the interarrival times process associated with a counting  process remain independent but  the exponential interarrival time distribution does not fit well into the observed data (cf. e.g. Chen et al. \cite{chen} and Wang et al. \cite{wang})  motivated Macheras \& Tzaninis \cite{mt3} to generalize the  Delbaen \& Haezendonck characterization for the more general renewal risk model  (also known as the Sparre Andersen model). 

In reality, risk portfolios are inhomogeneous and they can be seen as a mixture of smaller homogeneous portfolios which can be identified by the realization of a random variable (or a random vector) $\vT$. The latter interpretation leads to the notion of mixed counting processes. Based on that Meister \cite{me}, and latter on (in a more general setup) Lyberopoulos \& Macheras \cite{lm3}, generalized the  Delbaen \& Haezendonck characterization for the more general class of compound mixed Poisson processes, and applied their results to pricing CAT futures and to the theory of premium calculation principles, respectively.

Even though (compound) mixed Poisson processes provide a better practical model than the (compound) Poisson ones, their variance includes a quadratic term $t^2$, which is of importance only for large values of $t$. The latter, together with the fact that any (compound) mixed Poisson process is a Markov one, indicates the need for a proper generalization. In an effort to overcome these deficiencies, to allow more fluctuation  and to step away from a Markovian environment, we consider the mixed renewal risk model. Since the mixed renewal risk model is strictly more comprising than the mixed Poisson one,  the question  whether the  corresponding characterizations provided in  Delbaen \& Haezendonck \cite{dh}, Lyberopoulos \& Macheras \cite{lm3} and \cite{mt3}    can be extended to the more general mixed renewal risk model naturally arises, and it is precisely this problem the paper deals with. In particular, if the process $S$  is a compound mixed renewal process under the probability measure $P$, it would be interesting to characterize all probability measures $Q$ being progressively equivalent to $P$ and converting $S$ into a compound mixed Poisson process under $Q$.

Since conditioning is involved in the definition of compound mixed renewal processes (CMRPs for short), it is natural to expect that the notion of regular conditional probabilities plays an essential role for the investigation of CMRPs. For this reason, we first give in Section \ref{CRPPEM} a characterization of CMRPs via regular conditional probabilities, see Proposition \ref{crp}, which serves as a useful preparatory tool for the proofs of our results. Next, we prove the one direction of the desired characterization, see Proposition \ref{thm1}, which provides a characterization and an explicit calculation of Radon-Nikod\'{y}m derivatives $dQ/dP$ for well known cases in insurance mathematics, see Example \ref{ex1}. The arguments used in the proof of Proposition \ref{thm1} differ from those in the proofs of \cite{dh}, Proposition 2.1, and \cite{lm3}, Proposition 3.4, respectively, as a (mixed) renewal process does not have in general (conditionally) stationary and independent increments. Furthermore, it is worth noticing that we prove our results avoiding a well-known backward Markovisation technique appearing in e.g. \cite{scm}, Section 8.3.1.

 In Section \ref{char}, we first construct the canonical probability spaces admitting CMRPs, see Proposition \ref{prop1}. The latter result along with Proposition \ref{thm1}, is required for the desired characterization, in terms of regular conditional probabilities, of all those measures $Q$ which are progressively equivalent to an original probability measure $P$, such that a CMRP under $P$ remains a CMRP under $Q$, see Theorem \ref{thm!}. Note that the main results of \cite{mt3}, Theorem 3.1, and of \cite{lm3}, Theorem 4.3, follow as special instances of Theorem \ref{thm!}, see Remarks \ref{rem2} and \ref{rem3}, respectively. Another consequence of Theorem \ref{thm!} is that any CMRP can be converted into a compound mixed Poisson one through a change of measures technique, by choosing the ``correct" Radon-Nikod\'{y}m derivative, see Corollary \ref{cor1}. For  applications of Theorem \ref{thm!} to a characterization of equivalent martingale measures for CMRPs, and to the pricing of actuarial risks (premium calculation principles)  in an insurance market possessing the property of no free lunch with vanishing risk we refer to \cite{mt6a}. Finally, for applications of Theorem \ref{thm!} to the ruin problem for CMRPs we refer to \cite{t1}.  

\section{Preliminaries}\label{prel}

{\em Throughout this paper, unless stated otherwise, $(\vO,\vS,P)$ is a fixed but arbitrary probability space}.  The symbol  $\mathcal L^{1}(P)$   stands for the family of all real-valued $P$--integrable    functions on $\vO$. Functions that are $P$--a.s. equal are not identified. We denote by  $\sigma(\mathcal G)$  the $\sigma$--algebra generated by a family $\mathcal G$ of subsets of $\vO$. Given a topology  $\mf{T}$ on $\vO$ we write ${\mf B}(\vO)$ for its {\bf Borel $\sigma$--algebra} on $\vO$, i.e. the $\sigma$--algebra generated by $\mf{T}$. Our measure theoretic terminology is standard and generally follows  \cite{Co}.   For the definitions of real-valued random variables and random variables we refer to  \cite{Co}, p. 308.   We apply   notation $P_{X}:=P_X(\theta):={\bf{K}}(\theta)$ to mean that $X$ is distributed according to the law ${\bf{K}}(\theta)$, where $\theta\in D\subseteq\R^d$ ($d\in\N$)  is the parameter of the distribution. We denote again by ${\bf K}(\theta)$ the distribution function induced by the probability distribution ${\bf K}(\theta)$. Notation  ${\bf Ga}(b,a)$, where $a,b\in(0,\infty)$, stands for the law  of gamma  distribution (cf. e.g. \cite{Sc}, p. 180). In particular, ${\bf Ga}(b,1)={\bf Exp}(b)$ stands for the law of exponential distribution. For  two real-valued random variables $X$ and $Y$ we write $X=Y$ $P$--a.s. if $\{X\neq Y\}$ is a $P$--null set. If $A\subseteq\vO$, then $A^c:=\vO\setminus A$, while $\chi_A$ denotes the indicator (or characteristic) function of the set $A$.  For a map $f:D\longrightarrow E$ and for a non-empty set $A\subseteq D$ we denote  by $f\upharpoonright A$ the restriction of $f$ to $A$.  
 We write $\E_P[X\mid\F]$ for a version of a conditional expectation (under $P$) of $X\in\mathcal{L}^{1}(P)$ given a $\sigma$--subalgebra $\mathcal{F}$ of $\vS$. 
For $X:=\chi_E\in\mathcal{L}^1(P)$ with $E\in\vS$ we set $P(E\mid\mathcal{F}):=\E_P[\chi_E\mid\mathcal{F}]$. For the unexplained terminology of Probability and Risk Theory we refer to \cite{Sc}. \smallskip

Given two measurable spaces $(\vO,\vS)$ and $(\vY,H)$, a  function $k$ from $\vO\times H$ into $[0,1]$ is a {\bf $\vS$--$H$--Markov kernel} if it has the following properties:
\begin{itemize}
\item[{\bf(k1)}] The set-function $B\longmapsto k(\omega,B)$ is a probability measure on $H$ for any fixed $\omega\in\vO$.
\item[{\bf(k2)}] The function $\omega\longmapsto k(\omega,B)$ is $\vS$--measurable for any fixed $B\in{H}$.
\end{itemize}
In particular, given a real-valued random variable $X$ on $\vO$ and a $d$--dimensional random vector $\vT$ on $\vO$, a {\bf conditional distribution of $X$ over $\vT$} is a $\sigma(\vT)$--$\B$--Markov kernel denoted by $P_{X\mid\vT}:=P_{X\mid\sigma(\vT)}$ and satisfying for each $B\in\B$ condition
\[
P_{X\mid\vT}(\bullet,B)=P(X^{-1}[B]\mid\sigma(\vT))(\bullet)
\quad{P}\uph\sigma(\vT)\mbox{--a.s.}.
\]

Clearly, for every $\B_d$--$\B$--Markov kernel $k$, the map $K(\vT)$ from $\vO\times\B$ into $[0,1]$ defined by means of
\[ 
K(\vT)(\omega,B):=(k(\bullet,B)\circ\vT)(\omega)
\quad\mbox{for any}\;\;(\omega,B)\in \vO\times\B
\]
is a $\sigma(\vT)$--$\B$--Markov kernel. Then for $\theta=\vT(\omega)$ with $\omega\in\vO$ the probability measures $k(\theta,\bullet)$ are distributions on $\B$ and so we may write $\mathbf{K}(\theta)(\bullet)$ instead of $k(\theta,\bullet)$. Consequently, in this case $K(\vT)$ will be denoted by $\mathbf{K}(\vT)$.

For any real--valued random variables $X, Y$ on $\vO$ we say that $P_{X|\vT}$ and $P_{Y|\vT}$ are $P\upharpoonright\sigma(\vT)$--equivalent and we write $P_{X|\vT}=P_{Y|\vT}$ $P\upharpoonright\sigma(\vT)$--a.s.,  if there exists a $P$--null set $M\in\sigma(\vT)$ such that for any $\omega\notin M$ and $B\in \B$ the equality $P_{X|\vT}(\omega,B)=P_{Y|\vT}(\omega,B)$ holds true.\smallskip

A sequence $\{V_n\}_{n\in \N}$ of real-valued random variables on  $\vO$ is $P$--{\bf conditionally (stochastically) independent over $\sigma(\vT)$}, if for each $n\in\N$ with $n\geq 2$ we have
\[
P\big(\bigcap^{n}_{k=1} \{V_{k}\leq v_{k}\}\mid\sigma(\vT)\big)=\prod^{n}_{k=1} P\big(\{V_{k}\leq v_{k}\}\mid\sigma(\vT)\big)\quad P\upharpoonright\sigma(\vT)\mbox{--a.s.}.
\]
Furthermore, we say that  $\{V_n\}_{n\in \N}$ is {\bf $P$--conditionally identically distributed over $\sigma(\vT)$}, if 
\[
P\bigl(F\cap V_k^{-1}[B]\bigr)=P\bigl(F\cap V_m^{-1}[B]\bigr)
\]
whenever $k,m\in \N$, $F\in\sigma(\vT)$ and $B\in \B$.  We say that $\{V_n\}_{n\in \N}$ is
$P${\bf--conditionally (stochastically) independent or identically distributed given
	$\vT$}, if it is conditionally independent or identically distributed over the $\sigma$--algebra  $\sigma(\vT)$.\smallskip

{\em For the rest of the paper we simply write ``conditionally" in the place of ``conditionally given $\vT$” whenever conditioning refers to  $\vT$.}

\begin{rem} 
	\label{eqd}
	\normalfont
	If  the sequence $\{V_n\}_{n\in \N}$ is $P$--conditionally identically distributed, then it is  $P$--identically distributed.\smallskip
	
	In fact, for any $n\in\N$ and $B\in\B$ we have
	\[
	P_{V_n}(B)=\int P(V^{-1}_n[B]\mid\vT)\,dP=\int P(V^{-1}_1[B]\mid\vT)\,dP=P_{V_1}(B),
	\]
	where the second equality follows form the fact that $\{V_n\}_{n\in \N}$  is $P$--conditionally identically distributed.
\end{rem}

 {\em Henceforth,  unless stated otherwise, $(\vY,H):=((0,\infty),\B(\vY))$  and $\vT$ is a $d$--dimensional  random vector on $\vO$ with values on $D\subseteq\R^d$ ($d\in\N$).}

\section{Compound Mixed Renewal Processes and Progressively Equivalent Measures}\label{CRPPEM}

We first recall some additional background material, needed in this section.

A family $N:=\{N_t\}_{t\in\R_+}$ of  random variables from $(\vO,\vS)$ into $(\ov{\R}, \B(\ov\R))$ is called a {\bf counting process} if there exists a $P$--null set $\vO_N\in\vS$ such that the process $N$ restricted on $\vO\setminus\vO_N$ takes values in $\N_0\cup\{\infty\}$, has right-continuous paths, presents jumps of size (at most) one, vanishes at $t=0$ and increases to infinity. Without loss of generality we may and do assume, that $\vO_N=\emptyset$. Denote by $T:=\{T_n\}_{n\in\N_0}$ and $W:=\{W_n\}_{n\in\N}$  the {\bf{arrival process}} and {\bf{interarrival process}}, respectively (cf. e.g. \cite{Sc}, Section 1.1, page 6 for the definitions) associated with $N$. Note also that every arrival process induces a counting process, and vice versa (cf. e.g. \cite{Sc}, Theorem 2.1.1). \smallskip

Furthermore, let $X:=\{X_{n}\}_{n\in\mathbb N}$ be a sequence of positive  real-valued random variables on $\vO$, and for any $t\geq 0$ define
$$
S_t:=\begin{cases}  \sum^{N_t}_{k=1}X_k &\text{if}\;\; t>0; \\ 0 &\text{if}\;\; t=0. \end{cases}
$$
Accordingly, the sequence $X$ is said to be the {\bf claim size process}, and the family $S:=\{S_{t}\}_{t\in\mathbb R_{+}}$ of real-valued random variables on $\vO$ is said to be the {\bf aggregate  claims process induced by} the pair $(N,X)$.  Recall that a pair $(N,X)$ is called a {\bf risk process}, if $N$ is a counting process, $X$ is $P$--i.i.d. and the processes $N$ and $X$ are $P$--independent (see \cite{Sc}, Chapter 6, Section 6.1).\smallskip

The following definition has been introduced in \cite{lm6z3}, Definition 3.1, see also \cite{mt1}, Definition 3.2(b).
\begin{df}
	\label{mrp}
  	\normalfont
	A counting process $N$  is said to be a $P$--{\bf mixed renewal process with mixing parameter $\vT$ and  interarrival time conditional distribution $\bf{K}(\vT)$} (written $P$--MRP$(\bf{K}(\vT))$ for short), if the induced  interarrival process $W$ is $P$--conditionally independent and 
	$$
	\forall\;n\in\N\qquad  [P_{W_n\mid\vT}=\bf{K}(\vT)\quad P\upharpoonright\sigma(\vT)\text{-a.s.}].
	$$
	In particular, if the distribution of $\vT$ is degenerate at some point $\theta_0\in D$, then the counting process $N$ becomes  a $P$--{\em renewal process with  interarrival time distribution $\bf{K}(\theta_0)$} (written $P$--RP$(\bf{K}(\theta_0))$ for short).
\end{df}

Accordingly,  an aggregate claims process $S$ induced by a $P$--risk process $(N,X)$  such that $N$ is a $P$--MRP$(\bf{K}(\vT))$ is called  a {\bf compound mixed renewal process  with parameters ${\bf K}(\vT)$ and $P_{X_1}$} ($P$--CMRP$({\bf K}(\vT),P_{X_1})$ for short). In particular, if the distribution of $\vT$ is degenerate at $\theta_0\in D$ then $S$ is called a {\bf compound  renewal process with parameters ${\bf K}(\theta_0)$ and $P_{X_1}$} ($P$--CRP$({\bf K}(\theta_0),P_{X_1})$ for short).\smallskip

{\em Throughout what follows we denote again by ${\bf K}(\vT)$ and ${\bf K}(\theta)$ the conditional distribution function and the distribution function  induced by the conditional probability distribution ${\bf K}(\vT)$ and the  probability distribution ${\bf K}(\theta)$, respectively.}\smallskip

 The following conditions will be useful for our investigations:
 \begin{itemize}
 	\item[{\bf(a1)}] The pair $(W,X)$ is $P$--conditionally independent.
 	\item[{\bf(a2)}] The random vector $\vT$ and the process $X$ are $P$--(unconditionally) independent.
 \end{itemize}
 {\em Next, whenever condition (a1) and (a2) holds true we shall write that the quadruplet $(P,W,X,\vT)$ or (if no confusion arises) the probability measure $P$ satisfies (a1) and (a2), respectively}.\smallskip

 Since conditioning is involved in the definition of (compound) mixed renewal processes, it seems natural to investigate the relationship between such processes and regular conditional probabilities. To this purpose, we recall their definition.
 
 \begin{df}\label{rcp} 
 	\normalfont
 	Let $(\vY,H,R)$ be a probability probability space. A family $\{P_y\}_{y\in\vY}$ of probability measures on $\vS$ is called a {\bf regular conditional probability} (rcp for short) of $P$ over $R$  if
 	\begin{itemize}
 		\item[{\bf(d1)}] 
 		for each $E\in\vS$ the map $y\longmapsto P_y(E)$ is $H$--measurable;
 		\item[{\bf(d2)}]
 		$\int P_{y}(E)\,R(dy)=P(E)$ for each $E\in\vS$.
 	\end{itemize}
  \end{df}

We could use the term of {\em disintegration} instead, but it is better to reserve that term to the general case when $P_y$'s may be defined on different domains (see \cite{pa}).

 	If $f:\vO\longrightarrow\vY$ is an inverse-measure-preserving function (i.e. $P(f^{-1}(B))=R(B)$ for each $B\in{H}$), a rcp $\{P_{y}\}_{y\in\vY}$ of $P$ over $R$ is called {\bf consistent} with $f$ if, for each $B\in{H}$  the equality $P_{y}(f^{-1}(B))=1$ holds for $R$--almost every $y\in B$.\smallskip
 	
 	 We say that a rcp $\{P_{y}\}_{y\in\vY}$ of $P$ over $R$ consistent with $f$ is {\bf  essentially unique}, if for any other rcp $\{\wt P_{y}\}_{y\in\vY}$ of $P$ over $R$ consistent with $f$ there exists a $R$--null set $L\in H$ such that for any $y\notin{L}$ the equality $P_y=\wt P_y$ holds true.\smallskip

 {\em Remark.} 
 If $\vS$ is countably generated (cf. e.g. \cite{Co}, Section 3.4, page 102 for the definition) and $P$ is perfect (see \cite{fa}, p. 291 for the definition), then there always exists a rcp $\{P_{y}\}_{y\in\vY}$ of $P$ over $R$ consistent with any inverse-measure-preserving  map $f$ from $\vO$ into $\vY$ providing that $H$ is countably generated (see \cite{fa}, Theorems 6).
 So, in most cases appearing in applications (e.g. Polish spaces) rcps  as above always exist.\smallskip

 {\em From now on, unless stated otherwise, the family $\{P_\theta\}_{\theta\in D}$ is a rcp of $P$ over $P_\vT$ consistent with $\vT$.}\smallskip
 
Let $\T\subseteq\R_+$ with $0\in\T$. For a process $Y_\T:=\{Y_t\}_{t\in\T}$ denote by $\F_\T^Y:=\{\F^Y_t\}_{t\in\T}$ the canonical filtration of $Y_\T$. For  $\T=\R_+$ or $\T=\N_0$ we simply write $\F^Y$ instead of $\F^Y_{\R_+}$ or $\F^Y_{\N_0}$, respectively. Also, we write $\F:=\{\F_t\}_{t\in\R_+}$, where $\F_t:=\sigma(\F^S_t\cup\sigma(\vT))$, for the canonical filtration generated by $S$ and $\vT$, $\F^S_\infty:=\sigma(\bigcup_{t\in\R_+}\F^S_t)$ and $\F_\infty:=\sigma(\F^S_\infty\cup\sigma(\vT))$. \smallskip 
 
 The following characterization of compound mixed renewal processes in terms of rcps is of independent interest, as it allows us to convert a CMRP into a CRP via a suitable change of measures. It also plays a fundamental role in the characterization of progressively equivalent measures that preserve the  structure of a CMRP.
 
 \begin{prop}\label{crp} 
 	If  $P$ satisfies conditions (a1) and (a2), the following are equivalent:
 	\begin{enumerate}
 		\item $S$ is a $P$--CMRP$({\bf K}(\vT),P_{X_1})$;
 		\item there exists a $P_{\vT}$--null set $L_P\in\B(D)$ such that $S$ is a $P_\theta$--CRP$({\bf K}(\theta),P_{X_1})$ with $P_{X_1}=(P_{\theta})_{X_1}$ for every $\theta\notin L_P$.
 	\end{enumerate} 
 \end{prop}

\begin{proof} Assertion (i) is equivalent to the fact that $(N,X)$ is a risk process and that $N$ is a $P$-MRP($\mathbf{K}(\vT)$). But according to \cite{lm3}, Lemma 2.3, and \cite{lm6z3}, Proposition 3.8, we equivalently obtain two $P_\vT$--null sets   $L_{P,1}, L_{P,2}\in\B(D)$ such that the pair $(N,X)$ is a $P_\theta$--risk process for any $\theta\notin L_{P,1}$, and for any $\theta\notin L_{P,2}$ the process $N$ is a $P_\theta$--RP$({\bf K}(\theta))$. 
Let us fix an arbitrary $A\in\mathcal{F}_1^X$. Since $P$ satisfies condition (a2), applying \cite{lm1v}, Lemma 3.5, we get for any $B\in\mf{B}(D)$ that
\[
\int_BP_{\theta}(A)\,P_{\vT}(d\theta)=\int_BP(A)\,P_{\vT}(d\theta),
\]
or equivalently that there exists a $P_{\vT}$--null set $L_{P,3,A}\in\mf{B}(D)$ such that for any $\theta\notin{L}_{P,3,A}$ condition $P_{\theta}(A)=P(A)$ holds. But since $\mathcal{F}_1^X$ is countably generated, by a monotone class argument, we equivalently get a $P_{\vT}$--null set $L_{P,3}\in\mf{B}(D)$ such that for any $\theta\notin{L}_{P,3}$ condition $(P_{\theta})_{X_1}=P_{X_1}$ holds. Thus, putting $L_P:=L_{P,1}\cup L_{P,2}\cup L_{P,3}$, we equivalently get that $S$ is a $P_\theta$--CRP$({\bf K}(\theta),P_{X_1})$ for any $\theta\notin L_P$, completing in this way the proof.
\end{proof}   

   \begin{rems}\label{dis} 
 	\normalfont
 	{\bf (a)} An immediate consequence of Proposition \ref{crp}  is that, if $N$ is a $P$--MRP$({\bf K}(\vT))$ then the event of {\bf explosion} $E:=\{\sup_{n\in\N_0} T_n<\infty\}$ is a $P$--null set.\smallskip
 	
 	In fact, if  $N$ is a $P$--MRP$({\bf K}(\vT))$, then by Proposition \ref{crp} we have that $N$ is a $P_\theta$--RP$({\bf K}(\theta))$ for any $\theta\notin L_{P,2}$. But since  any renewal process has finite expectations (cf. e.g. \cite{se}, Proposition 4, page 101) and  any counting process with finite expectations has zero probability of explosion (cf. e.g. \cite{Sc}, Corollary 2.1.5), we obtain that $P_\theta(E)=0$ for any $\theta\notin L_{P,2}$; hence condition (d2) implies that
  \[
 	P(E)=\int_D P_\theta(E)\,P_\vT(d\theta)=0.
 	\] 
{\bf (b)} Let $S$ be the aggregate claims process induced by the counting process $N$ and the claim size process $X$. Fix on arbitrary $u\in\vY$ and $t\in\R_+$, and define the function $r_t^u:\vO\times{D}\longrightarrow\R$ by means of  
\begin{equation}\label{res1}
r_t^u(\omega,\theta):=u+c(\theta)\cdot{t}-S_t(\omega)\quad\mbox{for any}\quad (\omega,\theta)\in\vO\times{D},
\end{equation}
where $c$ is a positive $\B(D)$--measurable function. For arbitrary but fixed $\theta\in{D}$, the process $r^u(\theta):=\{r_t^u(\theta)\}_{t\in\R_+}$ defined by  $r_t^u(\theta):=r_t^u(\omega,\theta)$ for any $\omega\in\vO$, is called the {\bf reserve process} induced by the {\bf initial reserve} $u$, the {\bf premium intensity} or {\bf premium rate} $c(\theta)$ and the aggregate claims process $S$ (see \cite{Sc}, Section 7.1, pages 155-156 for the definition). The function $\psi_{\theta}$ defined by $\psi_{\theta}(u):=P_{\theta}(\{\inf{r}_t^u(\theta)<0\})$  is called the {\bf probability of ruin} for the reserve process $r^u(\theta)$ with respect to $P_\theta$ (see \cite{Sc}, Section 7.1, page 158 for the definition).

Define the real-valued functions $r_t^u(\vT)$ and $R_t^u$ on $\vO$ by means of  $
r_t^u(\vT)(\omega):=r_t^u(\omega,\vT(\omega))$ for any $\omega\in\vO$,  and $R_t^u:=r_t^u\circ(id_{\vO}\times\vT)$, respectively. The process $R^u:=\{R_t^u\}_{t\in\R_+}$ is called the {\bf reserve process} induced by the initial reserve $u$, the {\bf stochastic premium intensity} or {\bf stochastic premium rate} $c(\vT)$ and the aggregate claims process $S$. The function $\psi$ defined by 
 $\psi(u):=P(\{\inf{R}_t^u<0\})$ is called the {\bf probability of ruin} for the reserve process $R^u$ with respect to $P$.   	
   	
We claim that
\[
\psi(u)=\int_D\psi_{\theta}(u)P_{\vT}(d\theta)\quad\text{for any}\,\,u\in\vY.
\] 

 In fact, since $R_t^u=r_t^u(\vT)$ by definition, we have $\{\inf{R}_t^u<0\}=\{\inf{r}_t^u(\vT)<0\}$; hence applying \cite{lm1v}, Proposition 3.8 we get
\begin{align*}
\psi(u) &= P(\{\inf{r}_t^u(\vT)<0\})=\int_{\vO}\E_P[\chi_{(-\infty,0)}\circ\inf{r}_t^u(\vT)\mid\vT]dP\\
&= \int_{\vO}\E_{P_{\bullet}}[\chi_{(-\infty,0)}\circ\inf{r}_t^u(\bullet)]\circ\vT]dP=\int_D\E_{P_{\theta}}[\chi_{(-\infty,0)}\circ\inf{r}_t^u(\theta)]\circ{P}_{\vT}(d\theta)\\
&= \int_D\psi_{\theta}(u)P_{\vT}(d\theta).
\end{align*}
 \end{rems}

 \begin{ex}
 	\normalfont
 	\label{rup} 
 	Assume that $P$ satisfies condition (a1) and (a2), and let  $S$ be a $P$--CMRP$({\bf K}(\vT),P_{X_1})$  with  $P_{X_1}={\bf Exp}(\eta)$ and $\eta\in\vY$. Applying Proposition \ref{crp} we obtain a $P_\vT$--null set $L_P\in\B(D)$ such that $S$ is a $P_\theta$--CRP$({\bf K}(\theta),P_{X_1})$  for any $\theta\notin L_P$; hence we may apply \cite{rss}, Corollary 6.5.2, for $\delta=\eta$ and $\ga=R(\theta)$ in order to obtain 
 		\[
 		\psi_\theta(u)=\left(1-\frac{R(\theta)}{\eta}\right)\cdot e^{-R(\theta)\cdot u}\quad\text{for any}\,\,u\geq 0,
 		\]
 		where $R(\theta)$ is the unique positive solution of 
 		\[
 		\E_P[e^{r\cdot X}]\cdot \E_{P_\theta}[e^{-c\cdot r\cdot W_1}]=1.
 		\]
 		 Consequently, according to Remark \ref{dis}(b), the probability of ruin for a compound mixed renewal process  with    exponential claim size distributions is given by
 		\[
 		\psi(u)=\int_D \psi_\theta(u)\,P_\vT(d\theta)=\int_D \left(1-\frac{R(\theta)}{\eta}\right)\cdot e^{-R(\theta)\cdot u}\,P_\vT(d\theta)\quad\text{for any}\,\,u\geq 0.
 		\]
 \end{ex}

 In order to prove the main result of this section we need the following auxiliary lemmas.
 \begin{lem}
\label{lem1} 
Let $Q$ be a probability measure on $\vS$ such that $Q_\vT\sim P_\vT$, and let $\{Q_{\theta}\}_{\theta\in D}$ be a rcp of $Q$ over $Q_\vT$ consistent with $\vT$. The following hold true:
\begin{enumerate}
\item if $Q_{W_1}\sim P_{W_1}$, then there exists a $P_\vT$-null set $M\in\B(D)$ such that $(Q_\theta)_{W_1} \sim (P_\theta)_{W_1} $ for any $\theta\notin M$;
\item 	if $W$ is $P$-- and $Q$--conditionally i.i.d., then there exists a $P_\vT$--null set $\wt M\in\B(D)$, containing the $P_\vT$--null set $M$, and for any $\theta\notin \wt M$  there exists a $(P_\theta)_{W_1}$--a.s. positive Radon-Nikod\'{y}m derivative $r_{\theta}$ of $(Q_\theta)_{W_1}$ with respect to $(P_\theta)_{W_1}$, satisfying  for all $n\in\N_0$  condition
\begin{equation}
Q_\theta(E)=\E_{P_\theta}[\chi_E\cdot\prod_{j=1}^{n}\, r_{\theta}(W_j)]\quad\text{ for any }\,\, E\in\F^W_n.
\label{21}
\end{equation}
\end{enumerate} 
\end{lem}

\begin{proof}  Ad (i): Fix on arbitrary $B\in\B(\vY)$, and consider a $Q_{\vT}$--null set $M_{Q,B}\in\B(D)$ such that $(Q_{\theta})_{W_1}(B)=0$ for any $\theta\notin{M}_{Q,B}$. We then get
\[
Q_{W_1}(B)=\int (Q_\theta)_{W_1}(B)\,Q_\vT(d\theta)=\int_{M^c_{Q,B}}(Q_\theta)_{W_1}(B)\,Q_\vT(d\theta),
\]
implying $Q_{W_1}(B)=0$, or equivalently $P_{W_1}(B)=0$ by $Q_{W_1}\sim{P}_{W_1}$. It then follows by the property (d2) of $\{P_{\theta}\}_{\theta\in{D}}$ that there exists a $P_{\vT}$--null set $M_{P,B}\in\mf{B}(D)$ such that $(P_{\theta})_{W_1}(B)=0$ for any $\theta\notin{M}_{P,B}$. Replacing $Q$ with $P$, and assuming that there exists a $P_\vT$--null set $N_{P,B}\in\B(D)$ such that
$(P_\theta)_{W_1}(B)=0$ for any $\theta\notin N_{P,B}$
we conclude that there exists a $Q_\vT$--null set  $N_{Q,B}\in\B(D)$ such that $(Q_\theta)_{W_1}(B)=0$ for any $\theta\notin N_{Q,B}$. Put $M_B:=M_{Q,B}\cup M_{P,B}\cup N_{Q,B}\cup N_{P,B}\in\B(D)$. Clearly $M_B$ is a $P_{\vT}$--null set by $P_{\vT}\sim{Q}_{\vT}$. But since $\B(\vY)$  is countably generated, by a monotone class argument we can find a $P$--null set $M:=\bigcup_{B\in\G_{\B(\vY)}}M_B\in\sigma(\vT)$, where $\G_{\B(\vY)}$ is a countable generator of $\B(\vY)$ which is closed under finite intersections,  such that $(Q_\theta)_{W_1} \sim (P_\theta)_{W_1} $ for any $\theta\notin M$; hence (i) follows.

Ad (ii): Since $W$ is $P$-- and $Q$--conditionally i.i.d. and $Q_\vT\sim P_\vT$, it follows by \cite{lm6z3}, Lemma 3.7, that there exists a $P_\vT$--null set $M_1\in\B(D)$ such that $W$ is $P_\theta$-- and $Q_\theta$--i.i.d. for any $\theta\notin M_1$. Put $\wt M:= M_1\cup M$ and fix on arbitrary $n\in\N$. Assertion (i) implies that for any $\theta\notin \wt M$ there exists a $(P_\theta)_{W_1}$--a.s. positive Radon-Nikod\'{y}m derivative  $r_\theta$ of $(Q_{\theta})_{W_1}$ with respect to $(P_{\theta})_{W_1}$ such that 
\[
(Q_\theta)(W_n^{-1}[B])=\E_{P_\theta} \left[{\chi_{W_n^{-1}[B]}}\cdot r_{\theta}(W_n)\right] \quad \text{for any}\; B\in\mathfrak B(\vY).
\]

Putting $\wt{\C}^W_n:=\left\{\bigcap^{n}_{j=1} C_j : C_{j}\in\sigma(W_{j})\right\}$  we have that $\wt{\C}_n^W$ is a generator of $\mathcal{F}_n^W$, closed under finite intersections, and that any $C\in\wt{\C}_n^W$ satisfies condition  \eqref{21} by the  $P_\theta$--independence of $W$ for any $\theta\notin \wt M$. If by $\wt{\mathcal{D}}_n^W$ is denoted the family of all $E\in\mathcal{F}_n^W$ satisfying condition \eqref{21}, it can be easily shown that it is a Dynkin class containing $\wt{\C}_n^W$; hence by Dynkin Lemma we obtain $\wt{\D}_n^W=\F_n^W$, i.e. condition \eqref{21} holds.
\end{proof}

\begin{symbs}\label{symb2} 
\normalfont 
{\bf(a)}  Let $h$ be a real-valued, one to one $\B(\vY)$--measurable function. The class of all real-valued $\B(\vY)$--measurable functions $\ga$  such that $\E_{P}\left[h^{-1}\circ\ga\circ X_{1}\right]=1$ will be denoted by $\F_{P,h}:=\F_{P,X_1,h}$. The class of all real-valued  $\B(D)$--measurable functions $\xi$ on $D$ such that $P_{\vT}(\{\xi>0\})=1$  and $\E_P[\xi(\vT)]=1$ is denoted by $\mathcal{R}_+(D):=\mathcal{R}_+(D,\mathfrak{B}(D), P_{\vT})$.

{\bf (b)} Denote by $\mathfrak{M}^k(D)$ ($k\in\N$)  the class of all $\B(D)$--$\B(\R^k)$--measurable functions on $D$. In the special case $k=1$ we write $\mathfrak M(D):=\mathfrak M^1(D)$. By  $\mathfrak M_+(D)$ will be denoted the class of all positive $\B(D)$--measurable functions on $D$.  Two probability measures $P$ and $Q$ on $\vS$ are called {\bf progressively equivalent}, if they are equivalent (in the sense of absolute continuity) on each $\F_t$ (in symbols, $Q\uph\F_t\sim P\uph\F_t$ for any $t\geq 0$).  For each $\rho\in\mf{M}^k(D)$ the class of all probability measures $Q$ on $\vS$ which satisfy conditions (a1)  and  (a2),   are progressively equivalent to $P$, and such that $S$ is a $Q$--CMRP$({\bf \Lambda}(\rho(\vT)),Q_{X_1})$, is denoted by $\M_{S,{\bf\Lambda}(\rho(\vT))}:=\M_{S,{\bf \Lambda}(\rho(\vT)),P,{X_1}}$. In the special case $d=k$ and $\rho:=id_D$ we write $\M_{S,{\bf \Lambda}(\vT)}:=\M_{S,{\bf \Lambda}(\Ttheta)}$  for simplicity. 

{\bf (c)} 
For  given $\rho\in\mathfrak M^k(D)$ and  $\theta\in D$,  denote by ${\M}_{S,{\bf\Lambda}(\ttheta)}$  the class of all probability measures $Q_\theta$ on $\vS$, such that $Q_\theta\uph\F_t\sim P_\theta\uph\F_t$ for any $t\in\R_+$ and $S$ is a $Q_\theta$--CRP$({\bf \Lambda}(\ttheta),(Q_\theta)_{X_1})$.  
\end{symbs}

 {\em From now on, unless stated otherwise, $h$ is as in Notations \ref{symb2}(a), and $P\in\M_{S,{\bf K}(\vT)}$ is the initial probability measure under which $S$ is a $P$--CMRP$({\bf K}(\vT),P_{X_1})$.}  
 \smallskip

Recall that a  {\bf martingale in $\mathcal{L}^1(P)$ adapted to the filtration $\mathcal{Y}_{\mathbb{T}}$} or a $\mathcal{Y}_{\mathbb{T}}${\bf--martingale} in $\mathcal{L}^1(P)$ is a family $Y_\T:=\{Y_t\}_{t\in\T}$  of real--valued random variables in $\mathcal{L}^1(P)$ such that each $Y_t$ is  $\mathcal{Y}_t$--measurable and whenever $s\leq t$ in $\mathbb{T}$ condition $\int_A Y_s\,dP=\int_A Y_t\,dP$ holds  true for all $A\in\mathcal{Y}_s$.  A $\mathcal{Y}_\mathbb{T}$-martingale $\{Y_t\}_{t\in\T}$ in $\mathcal{L}^1(P)$ is  {\bf a.s. positive}  if $Y_t$ is $P$--a.s. positive for each $t\in\T$.  For $\mathcal{Y}_{\R_+}=\F$ we simply say that $Y:=Y_{\R_+}$ is a martingale in $\mathcal{L}^1(P)$.

\begin{lem}  
\label{lem2} 
For given $\rho\in\mathfrak{M}^k(D)$, let $Q\in\M_{S,{\bf\Lambda}(\rho(\vT))}$ and $\{Q_{\theta}\}_{\theta\in D}$ be a rcp of $Q$ over $Q_\vT$ consistent with $\vT$. Then there exist  a  $P_\vT$--null set $L_{\ast}\in\B(D)$ and a $P_{X_1}$--a.s. unique  function $\ga\in\F_{P,h}$ such that for any $\theta\notin L_{\ast}$
\begin{equation}
Q_\theta(A)=\E_{P_\theta}[\chi_A \cdot \wt M_t^{(\ga)}(\theta)]\quad\text{for all}\,\,\,0\leq u\leq t\,\,\text{and}\,\, A\in\F^S_{u},
\label{2}
\end{equation}
where 
\[
\wt M_t^{(\ga)}(\theta):=\left[\prod_{j=1}^{N_t}\,(h^{-1}\circ\ga\circ X_j)\cdot \frac{d\bf{\Lambda}(\ttheta)}{d\bf{K}(\theta)}(W_j)\right]\cdot\frac{1-{\bf{\Lambda}}(\ttheta)(t-T_{N_t})}{1-{\bf{K}}(\theta)(t-T_{N_t})},
 \]
and the family $\wt M^{(\ga)}(\theta):=\{\wt M^{(\ga)}_{t}(\theta)\}_{t\in\R_+}$ is an a.s. positive $\F^S$--martingale in $\mathcal{L}^1(P_\theta)$, satisfying condition $\E_{P_\theta}[\wt M^{(\ga)}_{t}(\theta)]=1$ for any $t\in\R_+$
\end{lem}

\begin{proof} We split the proof into the following steps.\medskip

 {\bf (a)} Conditions $Q_{\vT}\sim P_{\vT}$, $Q_{W_1}\sim P_{W_1}$ and $Q_{X_1}\sim P_{X_1}$ hold true.\smallskip

In fact, let us fix on arbitrary $t\in\R_+$. Since $P\uph\mathcal{F}_t\sim{Q}\uph\mathcal{F}_t$, we have $Q_{\vT}\sim{P}_{\vT}$ by $\sigma(\vT)\subseteq\mathcal{F}_t$, while by \cite{dh}, Lemma 2.1, we have $Q_{X_1}\uph\mathcal{F}_t^S\sim{P}_{X_1}\uph\mathcal{F}_t^S$ implying $Q_{X_1}\sim{P}_{X_1}$ by $\mathcal{F}_t^S\subseteq\mathcal{F}_t$. Finally, applying similar  arguments to those of the proof of Proposition 2.1 form \cite{mt3}, we get $Q_{W_n}\uph\mathcal{F}_t^S\sim{P}_{W_n}\uph\mathcal{F}_t^S$, implying $Q_{W_n}\sim{P}_{W_n}$ for any $n\in\N$ by $\mathcal{F}_t^S\subseteq\mathcal{F}_t$. The latter together with Remark \ref{eqd} yields $Q_{W_1}\sim P_{W_1}$. 
\medskip

{\bf (b)} There exists  a $P_\vT$--null set $\wt L:=L_P\cup{L}_Q\in\B(D)$, such that for any $\theta\notin \wt L$ the aggregate process $S$ is a $P_\theta$--CRP(${\bf K}(\theta),P_{X_1}$) with $P_{X_1}=(P_{\theta})_{X_1}$, and a $Q_\theta$--CRP(${\bf \Lambda}(\rho(\theta)),Q_{X_1})$ with $Q_{X_1}=(Q_\theta)_{X_1}$.\smallskip

In fact, due to Proposition \ref{crp} there exist a $P_{\vT}$--null set $L_P\in\mf{B}(D)$ and a $Q_{\vT}$--null set $L_Q\in\mf{B}(D)$ such that the process 
$S$ is a $P_{\theta}$--CRP($\mathbf{K}(\theta),P_{X_1}$) and a $Q_{\theta}$--CRP($\mathbf\Lambda(\rho(\theta)),Q_{X_1}$) with
$P_{X_1}=(P_{\theta})_{X_1}$ and $Q_{X_1}=(Q_{\theta})_{X_1}$ for any $\theta\notin{L}_P$ and any $\theta\notin{L}_Q$, respectively. 
Thus, taking into account that $P_{\vT}\sim{Q}_{\vT}$ by (a), we obtain that $\wt{L}\in\mf{B}(D)$ is a $P_{\vT}$--null set  such that for any $\theta\notin\wt{L}$ the conclusion of (b) holds.\medskip

{\bf (c)} There exists  a $P_\vT$--null set $\wt M\in\B(D)$, and for any $\theta\notin \wt M$  there exists a $(P_\theta)_{W_1}$--a.s. positive Radon-Nikod\'{y}m derivative $r_{\theta}$ of $(Q_\theta)_{W_1}$ with respect to $(P_\theta)_{W_1}$ satisfying condition \eqref{21}.\smallskip

In fact, since $P_{\vT}\sim{Q}_{\vT}$ by (a), and $W$ is $P$-- and $Q$--conditionally independent by assumption, we can apply Lemma \ref{lem1} in order to conclude the validity of (c).\medskip

{\bf (d)} There exists   a $P_{X_1}$--a.s. unique  function $\ga\in\F_{P,h}$, defined by means of $\ga:=h\circ f$, where $f$ is a  $P_{X_1}$--a.s. positive Radon-Nikod\'{y}m derivative  of $Q_{X_1}$ with respect to $P_{X_1}$, satisfying for every $n\in\N_0$ condition 
\begin{equation}
Q(E)=\E_P[\chi_E\cdot \prod_{j=1}^n (h^{-1}\circ \ga\circ X_j)]\quad\text{ for any }\,\, E\in\F^X_n.
\label{d1}
\end{equation}
In fact, since $P_{X_1}\sim{Q}_{X_1}$ by (a), we can apply \cite{mt3}, Lemma 2.2(a), in order to get (d). \medskip

{\bf (e)} There exists a  $P_\vT$--null set $L_{\ast}\in\B(D)$  such that for any $\theta\notin L_{\ast}$ the conclusion of the lemma holds true.

 In fact, put $L_\ast:=\wt L\cup \wt M$ and fix on an arbitrary $\theta\notin L_\ast$. Since $(Q_\theta)_{W_1}\sim (P_\theta)_{W_1}$ by (c), and $(Q_{\theta})_{X_1}\sim ({P}_{\theta})_{X_1}$ by (a) and (b), we can apply \cite{mt3}, Proposition 2.1, to complete the whole proof. 
\end{proof}

For a given aggregate claims process $S$ on $(\vO,\vS)$, in order to characterize the  progressively equivalent  measures that preserve the structure of a compound mixed renewal process (see  Theorem \ref{thm!}), one has to be able to characterize the Radon-Nikod\'{y}m derivatives $dQ/dP$. The next result provides such a characterization, as well as the one direction of our main result.

\begin{prop}\label{thm1} 
For given $\rho\in\mf M^k(D)$, let $Q$ be a probability measure on $\vS$ satisfying 
conditions (a1), (a2) and such that  $S$ is a $Q$--CMRP$({\bf\Lambda}(\rho(\vT)),Q_{X_1})$. If $\{Q_{\theta}\}_{\theta\in D}$ is a rcp of $Q$ over $Q_\vT$ consistent with $\vT$, then the following are all equivalent:
\begin{enumerate}
\item
$Q\uph\F_t\sim P\uph\F_t$ for any $t\in\R_+$;
\item
$Q_{X_{1}}\sim P_{X_{1}}$,  $Q_{W_{1}}\sim P_{W_{1}}$ and $Q_\vT\sim P_\vT$;
\item
$Q_\vT\sim{P}_\vT$ and there exist a $P_\vT$--null set $ L_{\ast\ast}\in\B(D)$  and a $P_{X_1}$--a.s. unique function $\ga\in\F_{P,h}$ such that for any $\theta\notin  L_{\ast\ast}$   
\begin{equation}
\tag{$RRM_\theta$}
Q_\theta(A)=\int_{A} \wt M^{(\ga)}_{t}(\theta)\,dP_\theta\quad\text{for all}\,\,\,0\leq u\leq t\,\,\text{and}\,\, A\in\F_{u}
\label{rcp2}
\end{equation} 
  and the family $\wt M^{(\ga)}(\theta):=\{\wt M^{(\ga)}_{t}(\theta)\}_{t\in\R_+}$, involved in condition \eqref{2},  is a $P_\theta$--a.s. positive martingale in $\mathcal{L}^1(P_\theta)$  satisfying condition $\E_{P_\theta}[\wt M^{(\ga)}_{t}(\theta)]=1$;
\item
there exist an essentially unique pair $(\ga,\xi)\in\F_{P,h}\times\mathcal{R}_+(D)$, where $\xi$ is a Radon-Nikod\'{y}m derivative  of $Q_\vT$ with respect to $P_\vT$, such that
\begin{equation}
\tag{$RRM_\xi$}
Q(A)=\int_{A} M^{(\ga)}_{t}(\vT)\,dP\quad\text{for all}\,\,\,0\leq u\leq t\,\,\text{and}\,\, A\in\F_{u},
\label{mart}
\end{equation}
where 
$$
M^{(\ga)}_{t}(\vT):= \xi(\vT)\cdot \wt M^{(\ga)}_{t}(\vT),
$$
and the family $M^{(\ga)}(\vT):=\{M^{(\ga)}_{t}(\vT)\}_{t\in\R_+}$  is a $P$--a.s. positive martingale in $\mathcal{L}^1(P)$ satisfying condition $\E_P[M^{(\ga)}_t(\vT)]=1$. 
\end{enumerate}
\end{prop}

\begin{proof} Fix on arbitrary $t\geq 0$ and let $u\in[0,t]$. Note that the implication (i) $\Longrightarrow$ (ii) follows by Lemma \ref{lem2}.\smallskip

Ad  (ii) $\Longrightarrow$ (iv):  Since $Q_\vT\sim P_\vT$, it follows  that there exists a $P_\vT$--a.s. positive Radon-Nikod\'{y}m derivative $\xi\in\mathcal R_+(D)$. Assumption $Q_{X_1}\sim P_{X_1}$ together with \cite{mt3}, Lemma 2.2(a), implies the existence of a $P_{X_1}$--a.s. positive unique function $\ga\in\F_{P,h}$ satisfying condition \eqref{d1}. We split the remaining proof of this part into several steps. \medskip

{\bf (a)} For any $A\in\F_u$ condition $ \E_P[\chi_A\cdot  \wt M^{(\ga)}_{u}(\vT)]=\E_{P_\vT}\left[\E_{P_\theta}[\chi_A\cdot  \wt M^{(\ga)}_{u}(\theta)]\right]$ holds true.\smallskip

In fact,   put $\mu:=P\circ (id_\vO\times \vT)^{-1}$ and for any $A\in\F_u$ consider the $\F_u\otimes \B(D)$--measurable map $v:=\chi_A\cdot \wt M^{(\ga)}_{u}:\vO\times D\longrightarrow\R$.  We will show first that $v\in\mathcal{L}^1(\mu)$. Since $\{P_{\theta}\}_{\theta\in D}$ is a rcp of $P$ over $P_\vT$ consistent with $\vT$, it follows by \cite{lm1v}, Proposition 3.7, that  it is a product rcp on $\vS$  for $\mu$ with respect to $P_\vT$ (see \cite{smm}, Definition 1.1, for the definition and the properties of a product rcp); hence
\begin{align*}
\int v\,d\mu&=\int_D\int_\vO v^\theta\,dP_\theta\,P_\vT(d\theta) =\int_D\int_\vO \chi_A\cdot \wt  M^{(\ga)}_{u}(\theta)\,dP_\theta\,P_\vT(d\theta)\\
&=\int_D\E_{P_\theta}[\chi_A\cdot \wt  M^{(\ga)}_{u}(\theta)]\,P_\vT(d\theta) \leq\int_D P_\vT(d\theta) =1,
\end{align*}
where the first equality follows by \cite{lm1v}, Remark 3.4(c), and the inequality follows from Lemma \ref{lem2}. Thus, applying \cite{lm1v}, Proposition 3.8(ii), for  $v\circ (id_\vO\times \vT)$, we get the desired conclusion.\medskip

{\bf (b)} Consider the family of sets
\[
\G_u:=\left\{\bigcap_{k=1}^{m}A_k : A_k\in\F^S_u\cup\sigma(\vT),\, m\in\N\right\}.
\]
Then condition  $Q(G)=\int_{G} M^{(\ga)}_{u}(\vT)\,dP\quad\text{for all}\,\,\,  G\in\G_{u}$,
 holds true.\smallskip

In fact,  for any $G\in\G_u$ there exist an integer $m\in\N$ and a finite sequence $\{A_k\}_{k\in\{1,\ldots,m\}}$ in $\F^S_u\cup\sigma(\vT)$, such that $G=\bigcap_{k=1}^{m}A_k$. Putting
$$
I_\vT:=\left\{k\in\{1,\ldots,m\} : A_k\in\sigma(\vT)\right\}\quad\text{and}\quad I_H:=\left\{k\in\{1,\ldots,m\} : A_k\in\F^S_u\setminus \sigma(\vT)\right\}
$$
we get  $I_\vT\cup I_H=\{1,\ldots,m\}$,  $\bigcap_{k\in I_\vT} A_k\in\sigma(\vT)$ and $\bigcap_{k\in I_H} A_k\in\F^S_u$.   Since $\bigcap_{k\in I_\vT} A_k\in\sigma(\vT)$, there exists a set $F\in\B(D)$ such that $\bigcap_{k\in I_\vT} A_k=\vT^{-1}[F]$. The latter, together with the consistency of $\{Q_\theta\}_{\theta\in D}$ with $\vT$, yields  
\begin{align*}
Q(G)&=Q(\bigcap_{k\in I_H} A_k\cap\vT^{-1}[F]) \stackrel{(d2)}{=}\int_{F\cap L^c_{\ast}} Q_\theta(\bigcap_{k\in I_H} A_k)\,Q_\vT(d\theta)\\
&=\int_{F\cap L^c_{\ast}}  \E_{P_\theta}[\chi_{\bigcap_{k\in I_H} A_k}\cdot\wt M^{(\ga)}_u(\theta)]\,Q_\vT(d\theta) =\int_{G} M^{(\ga)}_u(\vT)\,dP,
\end{align*}
where the third equality follows by Lemma \ref{lem2} and the fourth one by (a) along with the fact that $\xi$ is a Radon-Nikod\'{y}m derivative of $Q_\vT$ with respect to $P_\vT$.\medskip

{\bf (c)} The family $M^{(\ga)}(\vT)$ is a martingale in $\mathcal{L}^{1}(P)$, and for every $A\in\F_u$ condition \eqref{mart} holds true.\smallskip

 In fact, if we denote by $\D_u$ is the family of all $A\in\F_u$ satisfying condition \eqref{mart}, it follows by (b) that $\G_u\subseteq \D_u$, while an easy computation justifies that $\D_u$ is a Dynkin class. Therefore, by the Dynkin Lemma we get
\begin{equation}
Q(A)=\int_{A} M^{(\ga)}_{u}(\vT)\,dP\quad\text{ for any }\,\,\, A\in\F_{u},
\label{3}
\end{equation} 
implying
$$
\int_{A} M^{(\ga)}_{u}(\vT)\,dP=\int_{A} M^{(\ga)}_{t}(\vT)\,dP\quad\text{ for any }\,\,\, A\in\F_{u};
$$
hence $M^{(\ga)}(\vT)$ is a martingale in $\mathcal{L}^{1}(P)$. The latter along with 
condition \eqref{3} yields  
\eqref{mart}.\medskip

{\bf (d)}  $M^{(\ga)}(\vT)$  is $P$--a.s. positive  and satisfies condition $\E_P[M^{(\ga)}_{t}(\vT)]=1$.\medskip

In fact, first note that condition \eqref{mart} for $A=\vO$ yields
\[
\E_P[M^{(\ga)}_{t}(\vT)]=\int_{\vO} M^{(\ga)}_{t}(\vT)\,dP=Q(\vO)=1.
\]
Furthermore, since  $\wt M^{(\ga)}(\theta)$ is a $P_{\theta}$--a.s. positive $\F^S$--martingale in $\mathcal{L}^1(P_\theta)$ for any $\theta\notin L_\ast$, applying Lemma \ref{lem2} we obtain
\[
P(\{\wt M_t^{(\ga)}(\vT)>0\})=\int_{L_\ast^c} P_\theta(\{\wt M_t^{(\ga)}(\theta)>0\})\,P_\vT(d\theta)=1,
\]
implying, together with condition $P_\vT(\{\xi>0\})=1$, that $M^{(\ga)}(\vT)$  is $P$--a.s. positive. 
\medskip

Ad (iv) $\Longrightarrow$ (i):  The implication  (iv) $\Longrightarrow$ (i)  is immediate, since condition \eqref{mart} along with $P(\{M_t^{(\ga)}(\vT)>0\})=1$ implies (i).\medskip

Ad (iv) $\Longrightarrow$ (iii):  Fix on arbitrary $A\in\mathcal{F}_u, F\in\mf{B}(D)$ and $n\in\N_0$. We first establish the validity of the following claim: 

\begin{clm}
	\label{clm1}
	Condition 
	\[
	\F_u\cap\{N_u=n\}=\sigma(\F^X_n\cup\F^W_n\cup\sigma(\vT))\cap\{N_u=n\}
	\]
	holds true.
\end{clm}  
  \begin{proof}
	Since inclusion $\sigma(\F^X_n\cup\F^W_n\cup\sigma(\vT))\cap\{N_u=n\}\subseteq\F_u\cap\{N_u=n\}$
	is obvious, we only have to  show the inverse one. To this purpose, let $\G_u$, $I_\vT$ and $I_H$ be as in step (b).  For any $G\in\G_u$ there exist an integer $m\in\N$ and a finite sequence $\{A_k\}_{k\in\{1,\ldots,m\}}$ in $\F^S_u\cup\sigma(\vT)$, such that $G=\bigcap_{k=1}^{m}A_k$. Then there exists a set $B\in\B(D)$ such that $\bigcap_{k\in I_\vT} A_k=\vT^{-1}[B]$; hence 
	\[
	G\cap\{N_u=n\}=\left(\bigcap_{k\in I_H} A_k\cap\vT^{-1}[B]\right)\cap\{N_u=n\}=\left(\bigcap_{k\in I_H} A_k\cap\{N_u=n\}\right)\cap\vT^{-1}[B].
	\] 
	But since $\bigcap_{k\in I_H} A_k\in\F^S_u$, we get by \cite{mt3}, Lemma 2.1, that there exists a set $B_n\in\sigma(\F^X_n\cup\F^W_n)$ such that 
	\[
	\bigcap_{k\in I_H} A_k\cap\{N_u=n\}=B_n\cap\{N_u=n\};
	\]
	hence
	\[
	G\cap\{N_u=n\}=\left(B_n\cap\{N_u=n\}\right)\cap\vT^{-1}[B]=\left(B_n\cap\vT^{-1}[B]\right)\cap\{N_u=n\},
	\]
	implying
	\[
	\G_u\cap\{N_u=n\}\subseteq \sigma(\F^X_n\cup\F^W_n\cup\sigma(\vT))\cap\{N_u=n\},
	\]
	completing in this way the proof of the claim. 
\end{proof}

Since $\{Q_{\theta}\}_{\theta\in D}$ is a rcp of $Q$ over $Q_\vT$ consistent with $\vT$  we get
\begin{align*} 
Q\left((A\cap\{N_u=n\})\cap\vT^{-1}[F]\right)&\stackrel{(d2)}{=}\int_F Q_\theta(A\cap\{N_u=n\})\,Q_\vT(d\theta)\\
&=\int_F \xi(\theta)\cdot Q_\theta(A\cap\{N_u=n\})\,P_\vT(d\theta).
\end{align*}
Moreover, since $(A\cap\{N_u=n\})\cap\vT^{-1}[F]\in \F_u$, we get by (iv) that 
\begin{align*}
Q\left((A\cap\{N_u=n\})\cap\vT^{-1}[F]\right)&=\int_{A\cap\{N_u=n\})\cap\vT^{-1}[F]} M_t^{(\ga)}(\vT)\,dP\\
&= \int_{A\cap\{N_u=n\})\cap\vT^{-1}[F]} \xi(\vT)\cdot\wt M_t^{(\ga)}(\vT)\,dP\\
&= \int_{\vT^{-1}[F]} \xi(\vT)\cdot\E_P[\chi_{A\cap\{N_u=n\}}  \cdot\wt M_t^{(\ga)}(\vT)\mid\vT]\,dP\\
&= \int_{F} \xi(\theta)\cdot\E_{P_\theta}[\chi_{A\cap\{N_u=n\}}\cdot\wt M_t^{(\ga)}(\theta)]\,P_\vT(d\theta),
\end{align*}
where the last equality follows by \cite{lm1v}, Lemma 3.5; hence
\[
\int_F \xi(\theta)\cdot Q_\theta(A\cap\{N_u=n\})\,P_\vT(d\theta)=\int_{F} \xi(\theta)\cdot\E_{P_\theta}[\chi_{A\cap\{N_u=n\}}\cdot\wt M_t^{(\ga)}(\theta)]\,P_\vT(d\theta),
\]
implying that there exists a $P_\vT$--null set $L_{\ast\ast,A\cap\{N_u=n\}}\in\B(D)$, containing $L_\ast$, such that  
\begin{equation} 
Q_\theta(A\cap\{N_u=n\}) = \E_{P_\theta}[\chi_{A\cap\{N_u=n\}}\cdot\wt M_t^{(\ga)}(\theta)]
\label{equ1}
\end{equation}
for any $\theta\notin L_{\ast\ast,A\cap\{N_u=n\}}$.

But, due to Claim \ref{clm1}, the $\sigma$--algebra $\F_u\cap\{N_u=n\}$ is countably generated; hence there exists a countable generator $\mathcal{G}_{u,n}$ of its, closed under finite intersections. Putting 
\[
L_{\ast\ast}:=\bigcup\{L_{\ast\ast,A\cap\{N_u=n\}}: A\cap\{N_u=n\}\in\mathcal{G}_{u,n}\},
\]
we get a $P_{\vT}$--null set in $\mf{B}(D)$ such that condition \eqref{equ1} holds for any $A\cap\{N_u=n\}\in\mathcal{G}_{u,n}$ and $\theta\notin{L}_{\ast\ast}$. It is then easy to see that the family $\mathcal{D}_{u,n}$ of all elements of $\F_u\cap\{N_u=n\}$ satisfying condition \eqref{equ1} for any $\theta\notin{L}_{\ast\ast}$ is a Dynkin class containing $\mathcal{G}_{u,n}$. Thus, $\mathcal{D}_{u,n}=\F_u\cap\{N_u=n\}$ by the Dynkin Lemma, i.e. condition \eqref{equ1} holds  for any $A\in\F_u$ and $\theta\notin{L}_{\ast\ast}$. 

 The latter, along with   Proposition \ref{crp} as well as Remark \ref{dis}(a), yields
\[
	Q_{\theta}(A)=\sum_{n=0}^{\infty}Q_{\theta}(A\cap\{N_u=n\})=\E_{P_{\theta}}[\chi_A\cdot\wt{M}_t^{(\ga)}(\theta)] \quad \text{ for any }\,\, \theta\notin{L}_{\ast\ast}.
\]
 
Ad (iii) $\Longrightarrow$ (i):   Assuming assertion  (iii)  we get that $Q_\theta\uph\F_t\sim P_\theta\uph\F_t$ for any $t\geq 0$ and $\theta\notin L_{\ast\ast}$. But since $Q_\vT\sim{P}_\vT$ assertion  (i)  is immediate. 
\end{proof}

 Due to Proposition \ref{thm1}, under the weak conditions $Q_{X_{1}}\sim P_{X_{1}}$,  $Q_{W_{1}}\sim P_{W_{1}}$ and $Q_\vT\sim P_\vT$ the measures $P$ and $Q$ are equivalent on each  $\sigma$--algebra $\F_t$. We  will show here that this result does not, in general, hold true for the $\sigma$--algebra $\F_\infty$.

\begin{prop}\label{lem5} 
 Let  be given $Q\in\M_{S,{\bf \Lambda}(\Ttheta)}$ and $\{Q_{\theta}\}_{\theta\in D}$ a rcp  of $Q$ over $Q_\vT$ consistent with $\vT$. The following assertions hold true:   
\begin{enumerate}
\item 
if there exists a $P_\vT$--null set $\wh{L}_1$ in $\mf{B}(D)$ such that $P_\theta=Q_\theta$ for any $\theta\notin \wh{L}_1$, then the measures $P$ and $Q$ are equivalent on $\F_\infty$;  
\item 
if there exists a $P_\vT$--null set $\wh{L}_2$ in $\mf{B}(D)$  such that  $P_\theta\neq Q_\theta$  for any $\theta\notin \wh{L}_2$, then the measures $P$ and $Q$ are singular on $\F_\infty$, i.e. there exists a set $E\in\F_{\infty}$ such that $P(E)=0$ if and only if $Q(E)=1$.
\end{enumerate}
\end{prop}

\begin{proof} First note that since $Q\in\M_{S,{\bf \Lambda}(\Ttheta)}$, we get  $Q_\vT\sim P_\vT$ by Proposition \ref{thm1}. \smallskip

Ad (i): Assume that $P_\theta=Q_\theta$ for any $\theta\notin\wh{L}_1$, and consider the family 
\[
\G_\infty:=\left\{\bigcap_{k=1}^{m}A_k : A_k\in\F^S_\infty\cup\sigma(\vT),\, m\in\N\right\}.
\]
We then have $Q\uph\mathcal{G}_\infty\sim{P}\uph\mathcal{G}_\infty$.

In fact, let $A\in\G_\infty$ such that $Q(A)=0$. It then follows that there exists a number $m\in\N$  and a finite sequence $\{A_k\}_{k\in\{1,\ldots,m\}}$ in $\F^S_\infty\cup\sigma(\vT)$ such that $A=\bigcap_{k=1}^{m}A_k$. Putting
	\[
	I_\vT:=\left\{k\in\{1,\ldots,m\} : A_k\in\sigma(\vT)\right\}\quad\text{and}\quad I_H:=\left\{k\in\{1,\ldots,m\} : A_k\in\F^S_\infty\setminus \sigma(\vT)\right\}
	\]
	we get  $I_\vT\cup I_H=\{1,\ldots,m\}$,  $\bigcap_{k\in I_\vT} A_k\in\sigma(\vT)$ and $C:=\bigcap_{k\in I_H} A_k\in\F^S_\infty$.   Since $\bigcap_{k\in I_\vT} A_k\in\sigma(\vT)$, there exists a set $B\in\B(D)$ such that $\bigcap_{k\in I_\vT} A_k=\vT^{-1}[B]$, implying
\[
0=Q(C\cap\vT^{-1}[B])=\int_B Q_\theta(C)\,Q_\vT(d\theta).
\]

If $Q_\vT(B)=0$ then $P_\vT(B)=0$ by $Q_\vT\sim{P}_\vT$; hence $P(A)=0$.
\smallskip

If $Q_\vT(B)>0$ then there exists a $Q_\vT$--null set $L_{Q,C}\in\mf{B}(D)$ such that $Q_\theta(C)=0$ for any $\theta\notin{L}_{Q,C}$, implying that $P_\theta(A)=0$. In the same way, replacing $Q$ with $P$, we get that $P(A)=0$ implies $Q(A)=0$ for any $A\in\mathcal{G}_\infty$. Considering now the family $\mathcal{D}_\infty$ of all $A\in\F_\infty$ such that $Q(A)=0$ if and only if $P(A)=0$, we have that $\mathcal{D}_\infty$ is a Dynkin class containing $\mathcal{G}_\infty$. Thus, we may apply the Dynkin Lemma in order to conclude the validity of (i).

\medskip

 Ad (ii): By Proposition \ref{crp} along with Proposition \ref{thm1}, (i) $\Longrightarrow$ (ii), there exists a $P_\vT$-- and $Q_\vT$--null set $\wt{L}=L_P\cup{L}_Q$ such that the aggregate process $S$ is a $P_\theta$-CRP($\mathbf{K}(\theta),P_{X_1}$) with $P_{X_1}=(P_\theta)_{X_1}$ and a 
$Q_\theta$--CRP($\mathbf{\Lambda}(\rho(\theta)),Q_{X_1}$) with $Q_{X_1}=(Q_\theta)_{X_1}$ for any $\theta\notin\wt{L}$. Thus, applying \cite{mt3}, Remark 3.1, we get that for any $\theta\notin\wh{L}_2\cup\wt{L}$ the measures $P_\theta$ and $Q_\theta$ are singular on $\F_\infty^S$ ; hence on $\F_\infty$. Consequently, $P$ and $Q$ are singular on $\F_\infty$. 
\end{proof}

Proposition \ref{thm1} allows us to explicitly calculate Radon-Nikod\'{y}m derivatives for various cases appearing in applications. In the next example we consider the mixed Poisson process (cf. e.g. \cite{Sc}, page 87 for its  definition). A common choice for the distribution of $\vT$ in Risk Theory is the Gamma distribution. In the case of a mixed Poisson process that process is called P\'{o}lya-Lundberg process (cf. e.g. \cite{Sc}, page 100 for its definition and  basic properties). In order to present our first example recall  the  inverted Gamma distribution  with parameters   $a,b\in\vY$ (written ${\bf IG}(b,a)$ for short), i.e. 
$$
{\bf IG}(b,a)(B):=\int_B \frac{b^{a}}{\Gamma(a)}\cdot x^{-(a-1)}\cdot e^{-b/x}\,\lambda(dx)\quad\text{for any}\,\, B\in\B(\vY).
$$

\begin{ex}\label{ex1} 
	\normalfont
	Take $D:=\vY$, $\rho\in\mf{M}_+(\vY)$ defined by means of $\rho(x):=1/x$, $h:=\ln$, $\vT:\vO\longrightarrow \vY$, $P\in\M_{S,{\bf Exp}(\vT)}$ and $Q\in\M_{S,{\bf Exp}(\Ttheta)}$, such that  $P_\vT={\bf Ga}(b_1,a_1)$ and $Q_{\vT}={\bf IGa}(b_2,a_2)$ with $a_1,a_2,b_1,b_2>0$. 	By Proposition \ref{thm1}  there exists an essentially unique pair $(\ga,\xi)\in\F_{P,h}\times\mathcal{R}_+(\vY)$ such that 
	$$
	Q(A)=\int_A M^{(\ga)}_t(\vT) \,dP\quad\text{for all}\,\,\,0\leq u\leq t\,\,\text{and}\,\, A\in\F_{u},
	$$
	where $M^{(\ga)}_t(\vT):=\frac{b_2^{a_2}}{b_1^{a_1}}\cdot\frac{\Gamma(a_1)}{\Gamma(a_2)}\cdot\frac{e^{b_1\vT-b_2/\vT}}{\vT^{a_1+a_2}}\cdot e^{\sum^{N_{t}}_{j=1}\gamma(X_{j})}\cdot\left(\frac{\Ttheta}{\vT}\right)^{N_t}\cdot e^{-t (\Ttheta -\vT)}$.
\end{ex}

\section{The Characterization}\label{char}

Before we formulate the inverse of Proposition \ref{thm1}, i.e. that for a given pair $(\ga,\xi)\in\F_{P,h}\times\mathcal{R}_+(D)$ there exists a unique probability measure $Q\in\M_{S,{\bf \Lambda}(\Ttheta)}$, we have to prove the following result concerning the construction of compound mixed renewal  processes. To this purpose we recall the following notations concerning product probability spaces. \smallskip

By $(\vO\times\varXi,\vS\otimes{H},P\otimes{R})$ we denote  the product probability space of the probability spaces $(\vO,\vS,P)$ and $(\varXi,H,R)$.  If  $I$ is an arbitrary non-empty index set, we write $P_I$ for the product measure on $\vO^I$ and $\vS_I$ for its domain.\smallskip

{\em{Throughout what follows, we put $\wt{\vO}:=\vY^{\N}\times\vY^{\N}$, $\wt{\vS}:=\B(\wt{\vO})=\B(\vY)_{\N}\otimes\B(\vY)_{\N}$, $\vO:=\wt\vO\times D$ and $\vS:=\wt\vS\otimes\B(D)$ for simplicity.}}\smallskip

The following result enables us to construct {\em{canonical}} probability spaces admitting compound mixed renewal processes.

\begin{prop}\label{prop1} 
Let $\mu$ be a probability measure on $\B(D)$,  and  for any $n\in\N$ and fixed $\theta\in D$ let $P_{n}(\theta):={\bf{K}}(\theta)$ and $R_n:=R$ be probability measures on $\B(\vY)$. Assume that for any fixed $B\in\B(\vY)$ the function $\theta\longmapsto{\bf{K}}(\theta)(B)$ is $\B(D)$--measurable. Then there exist:
\begin{enumerate}
\item
a family $\{P_{\theta}\}_{\theta\in D}$ of probability measures $P_\theta:={\bf K}\left(\theta\right)_{\N}\otimes R_\N\otimes\delta_\theta$ on $\vS$, where $\delta_\theta$ is the Dirac measure on $\B(D)$ concentrated on $\theta$, and a probability measure $P$ on $\vS$ such that $\{P_{\theta}\}_{\theta\in D}$ is a rcp of $P$ over $\mu$ consistent with $\vT:=\pi_D$, where $\pi_D$ is the canonical projection from $\vO$ onto $D$, and $P_\vT=\mu$;
\item
a counting process $N$ being a $P$--MRP$({\bf K}(\vT))$, the   interarrival process $W$ of which satisfies condition $(P_\theta)_{W_n}=\mathbf{K}(\theta)$ for all $n\in\N$, a claim size process $X$ satisfying condition $P_{X_n}=R$ for all $n\in\N$,  such that the quadruplet $(P,W,X,\vT)$ satisfies conditions (a1) and (a2), and
 an aggregate claims process $S$ being a $P$--CMRP$({\bf K}(\vT),P_{X_1})$. 
\end{enumerate} 
\end{prop}

\begin{proof} Fix on arbitrary $\theta\in D$ and $n\in\N$, and consider the product probability space $(\wt\vO,\wt\vS,\wt P_\theta)$ constructed in \cite{mt3}, page 51, where $\wt P_\theta:=\left(\otimes_{n\in\N} P_n(\theta)\right)\otimes R_\N$. We split the proof into several steps. The first two steps establish the validity of (i), the remaining concern assertion (ii) of the proposition.\smallskip

{\bf (a)} For any fixed $F\in\wt\vS$ the function $\theta\longmapsto\wt{P}_\theta(F)$ is $\mf{B}(D)$--measurable.\smallskip

In fact, since by assumption, for any fixed $B\in\B(\vY)$ each function $\theta\longmapsto P_{n}(\theta)(B)$  is $\B(D)$--measurable, it follows by a monotone class argument that the same holds true for the function $\theta\longmapsto\wt{P}_{\theta}(F)$ for any fixed $F\in\wt{\vS}$. \medskip

{\bf (b)} Define the set-functions $\wt{P}:\wt{\vS}\longrightarrow \R_+$ and $P:\vS\longrightarrow\R_+$ by means of 
\[
\wt{P}(F):=\int \wt{P}_{\theta}(F)\,\mu(d\theta)\quad\text{for all}\quad F\in\wt{\vS}.
\]
and
\[
P(E):=\int \wt{P}_{\theta}(E^{\theta})\,\mu(d\theta)\quad\mbox{for each}\quad E\in\vS, 
\]
where $E^{\theta}:=\left\{\wt\omega\in\wt{\vO}:(\wt\omega,\theta)\in E\right\}$ is the $\theta$--section of $E$, respectively.
Then $\wt{P}$ and $P$ are probability measures on $\wt{\vS}$ and $\vS$, respectively, such that $\{P_{\theta}\}_{\theta\in D}$ is a rcp of $P$ over $\mu$ consistent with $\vT$ and $P_\vT=\mu$.\smallskip

In fact, obviously $\wt{P}$ and $P$ are probability measures on $\wt\vS$ and $\vS$, respectively.  
 It is easy to see that $\{\wt{P}_{\theta}\}_{\theta\in D}$ is a product rcp on $\wt{\vS}$ for $P$ with respect to $\mu$ (see \cite{smm}, Definition 1.1, for the definition and its properties). Put $P_{\theta}:=\wt{P}_{\theta}\otimes\delta_{\theta}$. Clearly, $P_{\theta}$ is a probability measure on $\vS$. So, we may apply \cite{lm5}, Proposition 3.5, to get that $\{P_{\theta}\}_{\theta\in D}$ is a rcp of $P$ over $\mu$ consistent with the canonical projection $\pi_D$ from $\vO$ onto $D$. Putting $\vT:=\pi_D$ we get $P_{\vT}=\mu$, completing in this way the proof of statement (i).\medskip

{\bf(c)} There exists a counting process $N$ and a claim size process $X$ such that the quadruplet $(P,W,X,\vT)$ satisfies conditions (a1), (a2), $W$ is $P_\theta$--i.i.d. and the pair $(N,X)$ is both a $P$-- and $P_\theta$-- risk process for any $\theta\in{D}$.\smallskip

In fact, denote by $\pi_{\vO\wt\vO}$ the canonical projection from $\vO$ onto $\wt\vO$, and by $\wt{W}_n$ and $\wt{X}_n$ the canonical projections from $\vO$ onto the $n$--coordinate of the first and the second factor of $\vO=\vY^{\N}\times\vY^{\N}\times{D}$, respectively. Put $W_{n}:=\wt{W}_{n}\circ \pi_{\vO\wt\vO} $ and $X_{n}:=\wt{X}_{n}\circ \pi_{\vO\wt\vO}$ and get 
\begin{equation}
\mathbf{K}(\theta)=(P_\theta)_{W_n}=(\wt{P}_\theta)_{W_n}\quad\mbox{and}\quad R=(P_\theta)_{X_n}=(\wt{P}_\theta)_{X_n}.
\label{equ4}
\end{equation}
Since $(\wt\vO,\wt\vS,\wt{P}_{\theta})$ is a product probability space and $\wt{W}_n$, $\wt{X}_n$ are the canonical projections, applying standard computations we get that the processes $W:=\{W_k\}_{k\in\N}$ and $X:=\{X_k\}_{k\in\N}$ are $\wt{P}_\theta$--independent and $\wt{P}_\theta$--mutually independent; hence the they are $P_\theta$--independent and $P_\theta$-mutually independent. Putting $T_k:=\sum_{m=1}^{k} W_m$ for any $k\in\N_0$ and $T:=\{T_k\}_{k\in\N_0}$, we obtain that $N:=\{N_t\}_{t\in\R_+}$ is the counting process induced by $T$ by means of $N_t:=\sum_{k=1}^{\infty}\chi_{\{T_k\leq t\}}$ for all $t\in\R_+$. The fact that $W$ and $X$ are $P_\theta$--mutually independent along with condition \eqref{equ4}, yields that the processes $N$ and $X$ are $P_\theta$--mutually independent. Thus, the pair $(N,X)$ is a $P_\theta$--risk process. 

Since $\{P_\theta\}_{\theta\in{D}}$ is a rcp of $P$ over $\mu$ consistent with $\vT$ by (b), applying \cite{lm1v}, Lemma 4.1, along with \cite{lm1err}, we obtain that $W$ and $X$ are $P$--conditionally independent, i.e. $P$ satisfies condition (a1). Furthermore, condition \eqref{equ4} again together with the fact that $\{P_\theta\}_{\theta\in{D}}$ is a rcp of $P$ over $\mu$ consistent with $\vT$ by (b), implies that $P_{X_n}=R$ and condition (a2) is satisfied by $P$. Thus, taking into account the fact that $X$ is $P_\theta$--i.i.d., we may apply \cite{lm3}, Lemma 3.3(ii), in order to conclude that the process $X$ is $P$--i.i.d.. Therefore the pair $(N,X)$ is a $P$--risk process.\medskip

{\bf (d)} The aggregate claims process $S$ induced by $(N,X)$ is a $P$--CMRP($\mathbf{K}(\vT),P_{X_1}$).\smallskip

In fact, since  the sequence $W$ is $P_\theta$--i.i.d. by (c), it follows that $N$ is a $P_\theta$--RP($\mathbf{K}(\theta))$, implying together with the fact that $(N,X)$ is a $P_\theta$--risk process by (b), that $S$ is a $P_\theta$--CRP$(\mathbf{K}(\theta),P_{X_1})$ with $P_{X_1}=(P_\theta)_{X_1}$, hence taking into account the fact that $P$ satisfies conditions (a1), (a2) by (c), we can apply Proposition \ref{crp} in order to get the conclusion of (d). Thus, assertion (ii) follows, completing the whole proof. 
\end{proof}

\begin{rem}\label{remcon} 
\normalfont
 Due to \cite{mt3}, Lemma 3.1, we get $\F_{\infty}^{S}=\F_{\infty}^{(W,X)}$, implying together with Proposition \ref{prop1} that $\vS=\F_\infty^{(W,X,\vT)}=\F_\infty$.
\end{rem}

\begin{symb}\label{symb3} 
For given $\rho\in\mathfrak M^k(D)$, let  $\theta\in D$ and let ${\bf K}(\theta)$ and ${\bf \Lambda}(\ttheta)$ be probability distributions on $\B(\vY)$.  For any $n\in\N_0$ the class of all likelihood ratios $g_n:=g_{\rho,n}:\vY^{n+1}\times D\longrightarrow\vY$ defined by means of 
$$
g_n(w_1,\ldots,w_n,t,\theta):=\Big[\prod_{j=1}^{n}\frac{d{\bf{\Lambda}}(\ttheta)}{d{\bf{K}}(\theta)}(w_j)\Big]\cdot\frac{1-{\bf{\Lambda}}(\ttheta)(t-w)}{1-{\bf{K}}(\theta)(t-w)}
$$
for any $(w_1,\ldots,w_n,t,\theta)\in\vY^{n+1}\times D$, where $w:=\sum^n_{j=1}w_j$, will be denoted by $\G_{n,\rho}$. Notation  $\G_{\rho}$ stands for the set  $\{g=\{g_n\}_{n\in\N_0} : g_n\in\G_{n,\rho}\,\,\,\text{for any}\,\,\, n\in\N_0\}$ of all sequences of elements of $\G_{n,\rho}$.
\end{symb}

{\em Throughout what follows ${\bf K}(\theta)$, ${\bf\Lambda}(\ttheta)$  and $g\in\G_{\rho}$  are as in Notation \ref{symb3}, and $P$, $\vT$, $\{P_{\theta}\}_{\theta\in D}$ and $S$ are as in Proposition \ref{prop1}.}

\begin{prop}\label{thm2} 
For given $\rho\in\mathfrak M^k(D)$ let  $(\ga,\xi)\in\F_{P,h}\times\mathcal R_+(D)$. Then for every $0\leq u\leq t$ and for all $A\in\F_{u}$ condition 
$$
Q(A)=\int_{A}\xi(\vT)\cdot \left[\prod_{j=1}^{N_t}\,(h^{-1}\circ\ga\circ X_j)\right]\cdot g(W_1,\ldots,W_{N_t},t,\vT)\,dP
$$
determines a unique probability measure $Q\in\M_{S,{\bf \Lambda}(\Ttheta)}$.
\end{prop}

\begin{proof} Let $(\ga,\xi)\in\F_{P,h}\times\mathcal R_+(D)$ and fix on arbitrary $t\geq 0$ and $n\in\N$. For any $\theta\in D$ define the set-functions $\wc\mu:\B(D)\longrightarrow\R$ and $\wc Q_n(\theta):\B(\vY)\longrightarrow \R$, by means of
	\[
	\wc\mu(F):=\E_P[\chi_{\vT^{-1}[F]}\cdot\xi(\vT)]\quad\text{ for any }\,\, F\in\B(D)
	\] 
	and 
	\[
	\wc Q_n(\theta)(B_1):=\E_{P_\theta}\left[\chi_{W_1^{-1}[B_1]}\cdot\left(\frac{d{\bf \Lambda} (\ttheta)}{d{\bf K} (\theta)}\circ W_1\right)\right] \quad\text{ for any }\,\, B_1\in\B(\vY),
	\]
	respectively. Also, consider the set-function $\wc R:\B(\vY)\longrightarrow \R$, defined by means of
\[
\wc R(B_2):=\E_{P}[\chi_{X^{-1}_{1}[B_2]} \cdot (h^{-1}\circ\ga\circ X_1)]\quad\text{ for any }\,\, B_2\in\B(\vY).
\]

 Clearly $\wc\mu$ and $\wc{Q}_n(\theta)$ are probability measures on $\B(D)$ and $\mf{B}(\vY)$, respectively, while $\wc{R}$ is a probability measure  by \cite{mt3}, Lemma 2.3(a). To show that $\wc{Q}_n(\theta)$ satisfies condition  
\begin{equation}
\wc{Q}_n(\theta)(B)=\mathbf{\Lambda}(\rho(\theta))(B)
\label{equ3}
\end{equation}
for any $B\in\mf{B}(\vY)$, put $\mathcal{C}^W:=\{(0,w]: w\in\vY\}$. Clearly, $\mathcal{C}^W$ is a generator of $\mf{B}(\vY)$, which is closed under finite intersections and satisfies condition \eqref{equ3}. Denoting by $\mathcal{D}^W$ the family of all elements of $\mf{B}(\vY)$ satisfying condition \eqref{equ3}, it can be easily shown that it is a Dynkin class containing $\mathcal{C}^W$; hence $\mathcal{D}^W=\mf{B}(\vY)$, i.e. condition \eqref{equ3} holds for any $B\in\mf{B}(\vY)$.

Thus, applying Proposition \ref{prop1} for $\wc\mu$, $\wc Q_n(\theta)$ and $\wc R$ in the place of $\mu$, $P_n(\theta)$ and $R$, respectively, we can construct a family $\{\wc Q_\theta\}_{\theta\in D}$ of probability measures on  $\vS$ defined by means of $\wc Q_\theta:={\bf\Lambda}(\ttheta)_\N\otimes\wc R_\N\otimes\delta_\theta$, a probability measure $\wc Q$ on $\vS$ satisfying conditions (a1) and (a2)  and  such that $\{\wc Q_\theta\}_{\theta\in D}$ is a rcp of $\wc Q$ over $\wc Q_\vT=\wc\mu$ consistent with $\vT$ and $S$ is a $\wc Q$--CMRP$({\bf\Lambda}(\Ttheta), \wc Q_{X_1})$ with $\wc Q_{X_1}=\wc R$. The latter, together with the definitions of $\wc\mu$, $\wc R$ and $\wc Q_n(\theta)$,  implies that $\wc Q_\vT\sim P_\vT$, $\wc Q_{X_1}\sim P_{X_1}$ and $(\wc Q_\theta)_{W_1}\sim (P_\theta)_{W_1}$ for any $\theta\in D$. But since  $(\wc Q_\theta)_{W_1}\sim (P_\theta)_{W_1}$ for any $\theta\in D$ and $\{P_\theta\}_{\theta\in D}$ and $\{\wc Q_\theta\}_{\theta\in D}$ are rcps of $P$ over $P_\vT$ and of $\wc Q$ over $\wc Q_\vT$, respectively, consistent with $\vT$, it follows easily that $\wc Q_{W_1}\sim P_{W_1}$.   Applying now Proposition \ref{thm1}, we get $\wc Q\uph\F_t\sim P\uph\F_t$, implying that $\wc Q\in\M_{S,{\bf \Lambda}(\Ttheta)}$,  or equivalently  
$$
\wc Q(A)=\int_{A}\xi(\vT)\cdot \prod_{j=1}^{N_t}\,(h^{-1}\circ\ga\circ X_j)\cdot g(W_1,\ldots,W_{N_t},t,\vT)\,dP
$$
for all $0\leq u\leq t$ and $A\in\F_{u}$. Thus $Q\uph\F_u=\wc{Q}\uph\F_u$ for all $u\in\R_+$; hence $Q\uph\wc\vS=\wc{Q}\uph\wc\vS$, where $\wc{\vS}:=\bigcup_{u\in\R_+} \F_u$, implying that $Q$ is $\sigma$--additive on $\wc\vS$  and that $\wc Q$ is the unique extension of $Q$ on $\vS=\sigma(\wc\vS)$, completing in this way the proof.
\end{proof} 

The following result is the desired characterization of progressively equivalent measures that preserve the   structure of a compound mixed renewal process.

\begin{thm}\label{thm!} 
	Let be given an arbitrary $\rho\in\mathfrak M^k(D)$. Then  the following hold true:
\begin{enumerate}
\item
for any  $Q\in\M_{S,{\bf \Lambda}(\Ttheta)}$ there exist an essentially unique pair $(\ga,\xi)\in\F_{P,h}\times\mathcal R_+(D)$, where $\xi$ is a Radon-Nikod\'{y}m derivative of $Q_\vT$ with respect to $P_\vT$, satisfying condition \eqref{mart};
\item
conversely, for  any pair $(\ga,\xi)\in\F_{P,h}\times\mathcal R_+(D)$ there exists  a unique probability measure $Q\in\M_{S,{\bf \Lambda}(\Ttheta)}$ determined by condition \eqref{mart}; 
\item in both cases (i) and (ii), there exists an essentially unique rcp $\{Q_{\theta}\}_{\theta\in D}$ of $Q$ over $Q_\vT$  consistent with $\vT$ and a $P_\vT$--null set $L_{\ast\ast}\in\B(D)$, satisfying for any $\theta\notin L_{\ast\ast}$ conditions $Q_\theta\in{\M}_{S,{\bf \Lambda}(\ttheta)}$ and \eqref{rcp2}. 
\end{enumerate}
\end{thm}

\begin{proof}  Assertions (i) and (ii) follow by Propositions \ref{thm1} and \ref{thm2}.\smallskip

Ad (iii): Since $\vO$ is a Polish space, according to the Remark following Definition \ref{rcp}, there always exists a rcp $\{Q_\theta\}_{\theta\in D}$ of $Q$ over $Q_\vT$ consistent with $\vT$; hence assertion
(iii) is an immediate consequence of Proposition \ref{thm1}.\end{proof}

\begin{rem}
\label{rem2}
\normalfont
 Theorem 3.9 of \cite{mt3} is an immediate consequence of Theorem \ref{thm!} if the distribution of $\vT$ under $P$ is degenerate at some $\theta_0\in D$.\smallskip

In fact,   assume that $\mu$ is a probability measure on $\B(D)$ such that $\mu:=\delta_{\theta_0}$ and for a fixed $\theta_0\in D$. According to Remark \ref{remcon},  we can construct a probability space $(\vO,\vS,P)$ such that $P_\vT(\{\theta_0\})=1$  and  $P\in\M_{S,{\bf K}(\vT)}$. Clearly, if we consider a probability measure $Q\in\M_{S,{\bf \Lambda}(\Ttheta)}$, then according to Proposition \ref{thm1}(ii) we obtain $Q_\vT\sim P_\vT$, implying  $Q_\vT(\{\theta_0\})=1$; hence without loss of generality we may assume that $\vT(\omega)=\theta_0$ for any $\omega\in\vO$. Consequently, we obtain  $\sigma(\vT)=\{\emptyset,\vO\}$ and  $\F_t=\F^S_t$ for any $t\in\R_+$; hence $\F_\infty=\F^S_\infty$. Moreover, note that in this case condition  \eqref{mart} is reduced to condition \eqref{rcp2}. Applying now Theorem \ref{thm!} for a degenerate $P_\vT$ we obtain Theorem 3.9 from \cite{mt3}. 
 \end{rem}

 \begin{symb}\label{mppnot} 
\normalfont
Denote by $\F_{P,\vT}:=\F_{P,\vT,X_1}$  the class of all real-valued $\B(\vY\times D)$--measurable functions $\be$ on $\vY\times D$ defined by means of $\be(x,\theta):=\ga(x)+\al(\theta)$ for any $x\in\vY$ and $\theta\in D$, where $\ga\in\F_{P,\ln}$ and $\al\in\mathfrak M(D)$.  Moreover, put $S_t^{(\be)}(\vT):=\sum^{N_t}_{j=1}\be(X_j,\vT)$ for any  $\be\in\F_{P,\vT}$.
 \end{symb} 

\begin{cor}\label{cor1} 
 If $\E_P[W_1|\vT]\in\vY\;\;\; P\uph\sigma(\vT)$--a.s. the following hold true:
\begin{enumerate}
\item for any pair $(\rho,Q)\in\mathfrak M_+(D)\times\M_{S,{\bf Exp}(\Ttheta)}$   there exists an essentially unique pair $(\be,\xi)\in\F_{P,\vT}\times\mathcal R_+(D)$, where $\xi$ is a Radon-Nikod\'{y}m derivative of $Q_\vT$ with respect to $P_\vT$, such that
\begin{equation}
\tag{$\ast$}
\ga =\ln f\quad\text{and}\quad \al(\vT)=\ln\rho(\vT)+\ln\E_{P}[W_1\mid\vT]\;\;\;P\uph\sigma(\vT)\mbox{a.s.},
\label{ast}
\end{equation} 
where $f$ is a $P_{X_1}$--a.s. positive Radon-Nikod\'{y}m derivative of $Q_{X_1}$ with respect to $P_{X_1}$, and 
\begin{equation}
\label{martPP}
\tag{$RPM_\xi$}
Q(A)=\int_A  M^{(\be)}_t(\vT)\,dP\quad\text{for all}\,\,\,0\leq u\leq t\,\,\text{and}\,\, A\in\F_{u},
\end{equation}
where  
 \[
 M^{(\be)}_t(\vT):= \xi(\vT)\cdot  \frac{e^{S_t^{(\be)}(\vT)- \Ttheta\cdot (t-T_{N_t})}\cdot(\Ttheta\cdot \E_P[W_1\mid\vT])^{-N_t}}{  1-{\bf{K}}(\vT) (t-T_{N_t})}\cdot \prod_{j=1}^{N_t}\frac{d{\bf Exp}(\Ttheta)}{d{\bf K}(\vT)}(W_j) ;
 \]

\item conversely, for any pair function $(\be,\xi)\in\F_{P,\vT}\times\mathcal R_+(D)$  there exist a unique pair $(\rho,Q)\in\mathfrak M_+(D)\times\M_{S,{\bf Exp}(\Ttheta)}$  determined by  \eqref{martPP} and  satisfying  condition  \eqref{ast}.

\item  
 in both cases (i) and (ii), there exists an essentially unique rcp $\{Q_{\theta}\}_{\theta\in D}$ of $Q$ over $Q_\vT$  consistent with $\vT$ and a $P_\vT$--null set $\wt L_{\ast\ast}\in\B(D)$, satisfying for any $\theta\notin \wt L_{\ast\ast}$ conditions $Q_\theta\in{\M}_{S,{\bf Exp}(\ttheta)}$,
\begin{equation}
\ga =\ln f\quad\text{and}\quad \ttheta=\frac{e^{\al(\theta)}}{\E_{P_\theta}[W_1]},
\tag{$\wt{\ast}$}
\label{*}
\end{equation}
 and
\begin{equation}
\tag{$RPM_\theta$}
Q_\theta(A)=\int_{A} \wt M^{(\be)}_{t}(\theta)\,dP_\theta\quad\text{for all}\,\,\,0\leq u\leq t\,\,\text{and}\,\, A\in\F_{u},
\label{rcp3}
\end{equation} 
where 
\[
\wt{M}^{(\be)}_t(\theta):= \frac{e^{S_t^{(\be)}(\theta)- \ttheta\cdot (t-T_{N_t})}\cdot(\ttheta\cdot \E_{P_\theta}[W_1])^{-N_t}}{  1-{\bf{K}}(\theta) (t-T_{N_t})}\cdot \prod_{j=1}^{N_t}\frac{d{\bf Exp}(\ttheta)}{d{\bf K}(\theta)}(W_j).
\]
 \end{enumerate}
 \end{cor}

\begin{proof} 
Fix on arbitrary $t\in\R_+$.\smallskip

Ad (i): If (i) holds then according to Theorem \ref{thm!}(i), there exist an essentially unique pair $(\ga,\xi)\in \F_{P,\ln}\times\mathcal R_+(D)$, with  $\ga=\ln f$, where $f$ is a Radon-Nikod\'{y}m derivative of $Q_{X_1}$ with respect to $P_{X_1}$, and $\xi$ is a  Radon-Nikod\'{y}m derivative  of $Q_\vT$ with respect to $P_\vT$,  such that
\begin{equation}
Q(A)=\int_A \xi(\vT)\cdot \frac{e^{\sum_{j=1}^{N_t} \ga(X_j)-\Ttheta\cdot(t-T_{N_t})}}{1-{\bf{K}}(\vT)(t-T_{N_t})}\cdot \prod_{j=1}^{N_t}\frac{d{\bf Exp}(\Ttheta)}{d{\bf K}(\vT)}(W_j)\,dP
\label{1111b}
\end{equation}
 for all $0\leq u\leq t$ and $A\in\F_u$. Putting  $\al(\vT)=\ln\rho(\vT)+\ln\E_{P}[W_1\mid\vT]$  $P\uph\sigma(\vT)$   and  $\be=\ga+\al$, we get that  $\be\in\F_{P,\vT}$ and   condition \eqref{ast} is valid. The latter, together with condition \eqref{1111b} implies  \eqref{martPP}.\smallskip
 
 Ad (ii): Consider the pair $(\be,\xi)\in\F_{P,\vT}\times\mathcal R_+(D)$ and  define $\rho\in\mathfrak M_+(D)$ by means of $\rho(\vT):=e^{\al(\vT)}/\E_{P}[W_1\mid\vT]$  $P\uph\sigma(\vT)$.   Then, condition \eqref{ast} is satisfied, while applying Theorem \ref{thm!}(ii) for $\ga=\be-\al$ we get  a unique probability measure $Q\in\M_{S,{\bf Exp}(\Ttheta)}$ satisfying condition \eqref{mart}  or equivalently condition \eqref{martPP}. \smallskip

Ad (iii): Since $\vO$ is a Polish space, according to the Remark following Definition \ref{rcp}, there always exists a rcp $\{Q_\theta\}_{\theta\in D}$ of $Q$ over $Q_\vT$ consistent with $\vT$. By \cite{lm1v}, Lemma 3.5,  there exists a $P_\vT$--null set $U_1\in\mf{B}(D)$ such that the second equality of condition \eqref{ast} is equivalent to $\al(\theta)=\ln\rho(\theta)+\ln\E_{P_\theta}[W_1]$ for any $\theta\notin{U}_1$. The latter along with Proposition \ref{thm1}, implies that there exists a $P_\vT$--null set $\wt L_{\ast\ast}:=L_{\ast\ast}\cup{U}_1\in\mf{B}(D)$ such that for any $\theta\notin{\wt L_{\ast\ast}}$ conditions  \eqref{*} and \eqref{rcp3} hold true. Clearly, condition \eqref{rcp3} implies that $Q_\theta\in{\M}_{S,{\bf Exp}(\ttheta)}$  for any $\theta\notin \wt L_{\ast\ast}$.
\end{proof}

\begin{rems}
\label{rem3} 
\normalfont
 {\bf (a)} For the special case $P\in\M_{S,{\bf Exp}(\vT)}$, Corollary \ref{cor1} yields the main result of Lyberopoulos \& Macheras \cite{lm3}, Theorem 4.3.\smallskip
 
  {\bf (b)} Fix on  $\ell\in\{1,2\} $. For given $\rho\in\mathfrak M^k(D)$,  define the classes 
\[
\M^\ell_{S,{\bf \Lambda}(\Ttheta)}:=\{Q\in\M_{S,{\bf \Lambda}(\Ttheta)}: \E_Q[X_1^\ell]<\infty\}
\]
and 
\[
\F^\ell_{P,h}:=\{\ga\in\F_{P,h}: \E_P[X_1^\ell\cdot (h^{-1}\circ\ga\circ X_1)]<\infty\}.
\] 
 
It can be easily seen that Theorem \ref{thm!} remains true, if
we replace the classes $\M_{S,{\bf \Lambda}(\Ttheta)}$ and  $\F_{P,h}$   by their subclasses $\M^\ell_{S,{\bf \Lambda}(\Ttheta)}$ and  $\F^\ell_{P,h}$, respectively. Consequently, Corollary \ref{cor1} remains true  if we replace d 
$\F_{P,\vT}$  by  it subclass
\[
\F^\ell_{P,\vT}:=\{\be=\ga+\al : \ga\in\F^\ell_{P,\ln}\,\,\text{and}\,\,\al\in\mathfrak M(D)\}.
\]

  	{\bf (c)} Assuming in Proposition \ref{prop1} that $\int_\vY x^\ell\,R(dx)<\infty$ for $\ell\in\{1,2\}$,  we get $\E_P[X^\ell_1]<\infty$, implying that $P\in\M^\ell_{S,{\bf K}(\vT)}$.
 	\smallskip

 	  {\bf (d)} Note that \cite{lm1v}, Lemma 3.5, remains true without the assumption $g\in\mathcal{L}^1(P)$ but only with the assumption that the integral $\int g\, dP$ is defined in $\R\cup\{-\infty,+\infty\}$.
 \end{rems}

 In the next examples, applying Corollary \ref{cor1}, we show how starting from a given pair $(\be,\xi)\in\F_{P,\vT}\times\mathcal R_+(D)$  we can construct a unique pair $(\rho,Q)\in\mathfrak M_+(D)\times \M_{S,{\bf Exp}(\rho(\vT))}$, converting an arbitrary compound mixed renewal process $S$ into a compound mixed Poisson one.\smallskip

 \textit{Throughout what follows assume that $\E_P[W_1|\vT]\in\vY\;\; P\uph\sigma(\vT)$--a.s..}\smallskip

In our first example we show how to find a probability measure $Q$ such  that $S$ is converted into a compound P\'{o}lya-Lundberg process under $Q$.

\begin{ex}
\label{exPL} 
\normalfont
Take $D:=\vY$, and let $\vT$ be a positive real-valued random variable. Assume that  $P_\vT$ is absolutely continuous with respect to the Lebesgue measure $\lambda\uph\B([0,1])$, and denote by $g$ the corresponding probability density function of $\vT$. Define the  function $\xi\in\mathfrak{M}_+(\vY)$ by means of 
\[
\xi(\theta):=\frac{b^a\cdot\theta^{a-1}\cdot e^{-b\cdot \theta}}{\Gamma(a)\cdot g(\theta)}\,\,\text{for any}\,\,\theta\in \vY,
\]
where $a,b\in\vY$ are   constants. Clearly $\E_P[\xi(\vT)]=1$, implying that $\xi\in\mathcal R_+(\vY)$.  Consider the function $\be(x,\theta):=\ga(x)+ \ln\left(\theta\cdot\E_{P_\theta}[W_1]\right)$ for any $(x,\theta)\in\vY^2$, with $\ga\in\F_{P,\ln}$; hence $\be\in\F_{P,\vT}$.  Applying now Corollary \ref{cor1}(ii) we get that there exists a unique pair $(\rho,Q)\in\mathfrak M_+(\vY)\times \M_{S,{\bf Exp}(\Ttheta)}$ satisfying conditions \eqref{ast} and \eqref{martPP}. In particular, it follows from condition \eqref{ast} that $\rho=id_\vY\;\; P_\vT$--a.s., and that 
\[
Q_\vT(B)=\E_P[\chi_{\vT^{-1}(B)}\cdot\xi(\vT)]=\int_B \frac{b^{a}\cdot\theta^{a-1}}{\Gamma(a)}\cdot e^{-b\cdot \theta}\,\lambda(d\theta)\quad\text{for any}\,\,B\in\B(\vY);
\]
hence the random variable $\vT$ satisfy condition $Q_{\vT}={\bf Ga}(b,a)$, implying that $S$ is a compound P\'{o}lya-Lundberg process under $Q$.
\end{ex}

The following example shows how to choose a probability measure $Q$ under which $S$ becomes a compound Poisson-Lognormal process. Recall 	the Lognormal distribution with parameters $\mu\in\R$ and $\sigma^2\in\vY$ (written ${\bf LN}(\mu,\sigma^2)$ for short) i.e.
$$
{\bf LN}(\mu,\sigma^2)(B):=\int_B  \frac {1}{\sqrt {2\pi }\cdot \sigma \cdot x}\ e^{-\frac {(\ln x-\mu )^{2}}{2\cdot\sigma^{2}}}\,\lambda(dx)\;\;\text{for any}\;\; B\in\B(\vY).
$$

\begin{ex}
\normalfont
Take $D:=\R$, and let $\vT$ be a real-valued random variable. Assume that  $P_\vT$ is absolutely continuous with respect to the Lebesgue measure $\lambda\uph\B([0,1])$, and denote by $g$ the corresponding probability density function of $\vT$. Define the  function $\xi\in\mathfrak{M}_+(\R)$ by means of 
\[
\xi(\theta):=\frac{1}{\sqrt{2\pi}\cdot \sigma\cdot g(\theta)}\cdot e^{-\frac{1}{2\cdot\sigma^2}\cdot(\theta-\mu)^2}\quad\text{for any}\,\,\theta\in\R,
\]
where $\mu\in\R$ and $\sigma\in\vY$ are   constants. Clearly $\E_P[\xi(\vT)]=1$, implying that $\xi\in\mathcal R_+(\R)$.  Consider the function $\be(x,\theta):=\ga(x)+ \ln\left(e^\theta\cdot\E_{P_\theta}[W_1]\right)$ for any $(x,\theta)\in\vY\times\R$, with $\ga\in\F_{P,\ln}$; hence $\be\in\F_{P,\vT}$.  Applying now Corollary \ref{cor1}(ii) we get that there exists a unique pair $(\rho,Q)\in\mathfrak M_+(\vY)\times \M_{S,{\bf Exp}(\Ttheta)}$ satisfying conditions \eqref{ast} and \eqref{martPP}. In particular, it follows from condition \eqref{ast} that $\rho(\theta)=e^\theta$ for $P_\vT$--a.a $\theta\in\R$, and that 
\[
Q_\vT(B)=\E_P[\chi_{\vT^{-1}(B)}\cdot\xi(\vT)]=\int_B \frac{1}{\sqrt{2\pi}\cdot \sigma}\cdot e^{-\frac{1}{2\cdot\sigma^2}\cdot(\theta-\mu)^2}\,\lambda(d\theta)\quad\text{for any}\,\,B\in\B;
\]
hence the random variable $\vT$ satisfy condition $Q_{\vT}={\bf N}(\mu,\sigma^2)$, implying that  $Q_{\rho(\vT)}={\bf LN}(\mu,\sigma^2)$ and that $S$ is a compound Poisson-Lognormal process under $Q$.
\end{ex}

In our final example we show how to transform a compound mixed renewal process into a compound Poisson-Beta one under a change of measures. Recall the Beta distribution with parameters $a,b\in\vY$ (written ${\bf Be}(a,b)$ for short) i.e.
$$
{\bf Be}(a,b)(B):=\int_B \frac{\Gamma(a+b)}{\Gamma(a)\cdot\Gamma(b)}\cdot x^{a-1}\cdot(1-x)^{b-1}\,\lambda(dx)\;\;\text{for any}\;\; B\in\B((0,1)).
$$

\begin{ex}
\normalfont
Take $D:=\vY^2$, and let $\vT:=(\vT_1,\vT_2)$ be a $2$--dimensional random vector, where $\vT_1$ and $\vT_2$ are two positive real-valued random variables on $\vO$. Assume also that  $P_\vT$ is absolutely continuous with respect to the Lebesgue product measure $\lambda_2\uph\B([0,1]^2)$, and denote by $g:\vY^2\longrightarrow\vY$ the corresponding probability density function of $\vT$. Define the function $\xi\in\mathfrak{M}_+(\vY^2)$ by means of
$$
\xi(\theta):=\xi(\theta_1,\theta_2):=\frac{a^{b_1+b_2}\cdot\theta_1^{b_1-1}\cdot\theta_2^{b_2-1}}{\Gamma(b_1)\cdot\Gamma(b_2)\cdot g(\theta_1,\theta_2)}\cdot e^{-a(\theta_1+\theta_2)}\quad\text{for any}\,\,(\theta_1,\theta_2)\in\vY^2
$$
where $a,b_1,b_2\in\vY$ are   constants. Clearly $\E_P[\xi(\vT)]=1$, implying that $\xi\in\mathcal R_+(\vY^2)$. Consider the function $\be(x,\theta):=\ga(x)+ \ln\frac{\theta_1  \cdot \E_{P_\theta}[W_1]}{\theta_1+\theta_2}$ for any $(x,\theta)\in\vY\times\vY^2$, with $\ga\in\F_{P,\ln}$; hence $\be\in\F_{P,\vT}$.  Applying now Corollary \ref{cor1}(ii) we get that there exists a unique pair $(\rho,Q)\in\mathfrak M_+(\vY)\times \M_{S,{\bf Exp}(\Ttheta)}$ satisfying conditions \eqref{ast} and \eqref{martPP}. In particular, it follows from condition \eqref{ast} that $\rho(\theta)=\frac{\theta_1}{\theta_1+\theta_2}$ for $P_\vT$--a.a. $\theta=(\theta_1,\theta_2)\in\vY^2$ and that 
\begin{align*}
Q_\vT(B_1\times B_2)&=\E_P[\chi_{\vT^{-1}(B_1\times B_2)}\cdot\xi(\vT)]\\
&=\int_{B_1\times B_2} \frac{a^{b_1+b_2}\cdot\theta_1^{b_1-1}\cdot\theta_2^{b_2-1}}{\Gamma(b_1)\cdot\Gamma(b_2)}\cdot e^{-a(\theta_1+\theta_2)}\,\lambda_2(d\theta)\\
&=\int_{B_1\times B_2} \frac{a^{b_1+b_2}\cdot\theta_1^{b_1-1}\cdot\theta_2^{b_2-1}}{\Gamma(b_1)\cdot\Gamma(b_2)}\cdot e^{-a(\theta_1+\theta_2)}\,\lambda(d\theta_2)\,\lambda(d\theta_1)\\
&=\left(\int_{B_1} \frac{a^{b_1}\theta_1^{b_1-1}}{\Gamma(b_1)}\cdot e^{-a\theta_1}\,\lambda(d\theta_1)\right)\cdot\left(\int_{B_2} \frac{a^{b_2}\theta_2^{b_2-1}}{\Gamma(b_2)}\cdot e^{-a\theta_2}\,\lambda(d\theta_2)\right)\\
&=(Q_{\vT_1}\otimes Q_{\vT_2})(B_1\times B_2),
\end{align*}
where $B_1,B_2\in\B(\vY)$; hence $Q_{\vT_1}={\bf Ga}(b_1,a)$ and $Q_{\vT_2}={\bf Ga}(b_2,a)$ and  $\vT_1$ and $\vT_2$ are $Q$--independent, implying that  $Q_{\rho(\vT)}={\bf Be}(b_1,b_2)$ and that $S$ is a compound Poisson-Beta process. 
\end{ex}

Another consequence of Theorem \ref{thm!} is the following result which shows that the martingales $L^r:=\{L_t^r\}_{t\in\R_+}$ and the measures $Q^r$ appearing in \cite{scm}, Lemma 8.4, are special instances of the martingales $\wt{M}^{(\ga)}(\theta)$ and the measures $Q_{\theta}$, respectively, of Theorem \ref{thm!}.  Note that Lemma 8.4 of \cite{scm}, was first proven by Dasios \& Embrechts in \cite{daem}, Theorem 10.  To do it, we first need to establish the validity of the following lemma.

\begin{lem}\label{lem35} 
For any $r\in\R_+$ such that $\E_P[e^{rX_1}]<\infty$ and for any $\theta\in{L}_P^c$, let $\kappa_{\theta}(r)$ be the unique solution to the equation  
\begin{equation}\label{lem35a}
M_{X_1}(r)\cdot(M_{\theta})_{W_1}\bigl(-\kappa_{\theta}(r)-c(\theta)\cdot{r}\bigr)=1,
\end{equation}
where $M_{X_1}$ and $(M_{\theta})_{W_1}$ are the moment generating function of $X_1$ and $W_1$ under the measures $P$ and $P_{\theta}$, respectively. (Such a solution exists by e.g. \cite{rss}, Lemma 11.5.1(a)). Define the function $\kappa:D\times\R_+\longrightarrow\R$ by means of $\kappa(\theta,r):=\kappa_{\theta}(r)$ for any $(\theta,r)\in{D}\times\R_+$,
and for fixed $r\in\R_+$ denote by $\kappa_{\vT}(r)$ the random variable defined by the formula
$\kappa_{\vT}(r)(\omega):=\kappa_{\vT(\omega)}(r)$ for any $\omega\in\vO$. Then $\kappa_{\vT}(r)$ is the $P\uph\sigma(\vT)$--a.s. unique solution to the equation 
\begin{equation}\label{lem35b}
M_{X_1}(r)\cdot \E_P\bigl[e^{-\bigl(\kappa_{\vT}(r)+c(\vT)\cdot{r}\bigr)W_1}\mid\vT\bigr]=1\quad{P}\uph\sigma(\vT)\text{--a.s.}.
\end{equation}
\end{lem}

\begin{proof}
Note that, according to Proposition \ref{crp}, there exists a $P_{\vT}$-null set $L_P\in\mf{B}(D)$ such that the process $S$ is a $P_{\theta}$--CRP$\bigl(\mathbf{K}(\theta),P_{X_1}\bigr)$ with $P_{X_1}=(P_{\theta})_{X_1}$ for any $\theta\notin{L}_P$. Next we need to establish the following claim.
\begin{clm}\label{lem35c} 
Let $u$ be a $[-\infty,+\infty]$--valued function on $D\times\vY$ such that the integral $\int{u}\,dM$, where $M:=P_{(W_1,\vT)}$, is defined in $[-\infty,+\infty]$, and let $g:=u(W_1,\vT)$. The following hold true:
\begin{enumerate}
\item
The integral $\int\int{u}(w,\theta)\,(P_{\theta})_{W_1}(dw)\,P_{\vT}(d\theta)$ is defined and equal to $\int{u}\,dM$;
\item
$\E_P[g\mid\vT]=\E_{P_{\bullet}}[u^{\bullet}]\circ\vT\quad{P}\uph\sigma(\vT)\text{--a.s.}$;
\item
$\int{g}\,dP=\int\int{u}^{\bullet}\,dP_{\bullet}\,dP_{\vT}$.
\end{enumerate}
\end{clm}

\begin{proof}
Ad (i): For any $\theta\in{D}$ define the probability measure $\wt{P}_{\theta}$ on $\mf{B}(\vY)$ by means of
\[
\wt{P}_{\theta}(A):=(P_{\theta})_{W_1}(A)\quad\mbox{for any}\quad A\in\mf{B}(\vY).
\]
It follows easily that $\{\wt{P}_{\theta}\}_{\theta\in{D}}$ is a product rcp on $\mf{B}(\vY)$ for $M$ with respect to $P_{\vT}$ (see \cite{smm}, Definition 1.1 for the definition). By \cite{lm1v}, Remark 3.4(b), we have 
\[
M(E)=\int\wt{P}_{\theta}(E_{\theta})\,P_{\vT}(d\theta)\quad\mbox{for any}\quad E\in\mf{B}(D\times\vY);
\]
hence (i) holds for $u=\chi_E$. Applying now the standard methods of integration theory one can show that (i) holds first for non-negative $[0,+\infty]$--valued $\mf{B}(D\times\vY)$--measurable 
functions $u$, and then for general functions 
$u=u^+-u^-$, where $u^+$ and $u^-$ are the positive and negative parts of $u$, respectively.

The proof of the statements (ii) and (iii) follow in a similar way as that in \cite{lm1v}, Proposition 3.8, by replacing $id_{\vO}\times{f}$ with $W_1\times \vT$.
\end{proof}

 Due to Claim \ref{lem35c} along with Proposition \ref{crp}, we get that equation \eqref{lem35a} is equivalent to equation \eqref{lem35b}; hence \eqref{lem35b} holds and $\kappa_{\vT}(r)$ is the $P\uph\sigma(\vT)$--a.s. unique solution to \eqref{lem35b}.
\end{proof} 

\begin{prop}\label{Schmidli} 
For any $r\in\R_+$ such that $\E_P[e^{rX_1}]  <\infty$,  and   any $\theta\notin{L}_P$, let $\kappa_{\theta}(r)$ be the unique solution to the equation \eqref{lem35a}, and let $\kappa_{\vT}(r)$ be as in Lemma \ref{lem35}. Fix on arbitrary $r\in\R_+$ as above. 
\begin{enumerate} 
\item
Assume that $Q:=Q^r\in\mathcal{M}_{S,\mathbf{\Lambda}(\rho(\vT))}$ with  	
 \[
Q_{X_1}(B_1):=Q_{X_1}^r(B_1):=\frac{\E_{P}[\chi_{X_1^{-1}[B_1]}\cdot e^{r\cdot X_1}]}{\E_{P}[e^{r \cdot X_1}]}\quad\text{ for any }\,\,B_1\in\B(\vY)
\]
and 
\[
Q_{W_1\mid\vT}(B_2):= Q_{W_1\mid\vT}^r(B_2):={\bf \Lambda (\Ttheta)}:=\frac{\E_{P}[\chi_{W_1^{-1}[B_2]}\cdot e^{-(\kappa_\vT(r)+c(\vT)\cdot r)\cdot W_1}\mid\vT]}{\E_{P}[e^{-(\kappa_\vT(r)+c(\vT)\cdot r)\cdot W_1}\mid\vT]}
\]
for any $B_2\in\B(\vY)$, where the last three equalities hold $P\uph\sigma(\vT)$--a.s.. There exists an essentially unique pair $(\ga,\xi)\in\F_{P,\ln}\times \mathcal R_+(D)$ satisfying  conditions $\gamma(x)=r\cdot x-\ln\E_{P}[e^{r \cdot X_1}]$ for any $x\in\vY$,  and \eqref{mart}, with  
\begin{equation}
M_t^{(\ga)}(\vT)=\xi(\vT)\cdot\wt{M}_t^{(\ga)}(\vT)\;\; P\uph\sigma(\vT)-\mbox{a.s.},
\label{sch1}
\end{equation}
where
\[
\wt{M}_t^{(\ga)}(\vT)=e^{r\cdot S_t-(\kappa_\vT(r)+c(\vT)\cdot r)\cdot T_{N_t} +\ln \E_P[e^{r\cdot X_1}]}\cdot \frac{\int^{\infty}_{t-T_{N_t}} e^{-(\kappa_\vT(r)+c(\vT)\cdot r)\cdot w} \, P_{W_1\mid\vT}(dw)}{1-{\bf{K}}(\vT )(t-T_{N_t})}.
\]  	 
\item
Conversely, for any pair $(\ga,\xi)$ with $\xi\in\mathcal{R}_+(D)$ and $\gamma(x):=r\cdot x-\ln\E_{P}[e^{r \cdot X_1}]$ for any $x\in\vY$ there exists a unique probability measure $Q:=Q^r\in\M_{S,{\bf\Lambda}(\rho(\vT))}$   determined by  condition \eqref{mart} with $M^{(\ga)}(\vT)$ fulfilling condition \eqref{sch1}.
\item
In both cases (i) and (ii) there exist an essentially unique rcp $\{Q_{\theta}\}_{\theta\in D}:=\{Q_{\theta}^r\}_{\theta\in{D}}$ of $Q$ over $Q_\vT$  consistent with $\vT$ and a $P_\vT$-null set $L_{\ast\ast}\in\B(D)$, satisfying for any $\theta\notin L_{\ast\ast}$ conditions $Q_\theta\in{\M}_{S,{\bf \Lambda}(\ttheta)}$ and \eqref{rcp2}, with
\[
\wt M_t^{(\ga)}(\theta) =e^{r\cdot S_t-(\kappa_\theta(r)+c(\theta)\cdot r)\cdot T_{N_t} +\ln \E_P[e^{r\cdot X_1}]}\cdot \frac{\int^{\infty}_{t-T_{N_t}} e^{-(\kappa_\theta(r)+c(\theta)\cdot r)\cdot w} \, (P_\theta)_{W_1}(dw)}{1-{\bf{K}}(\theta)(t-T_{N_t})}.
\]
\end{enumerate}

In particular, if $P_{W_1}$ is absolutely continuous with respect to the Lebesgue measure $\lambda$ restricted to $\mf{B}([0,1])$, then the martingale $L^r(\theta):=\{L^r_t(\theta)\}_{t\in\R_+}$ for $r\in\R_+$, appearing in \cite{scm}, Lemma 8.4, coincides with the martingale $\wt{M}^{(\ga)}(\theta)$ for any $\theta\notin{L}_{\ast\ast}$, and for any $t\in\R_+$ condition 
\[
M_t^{(\ga)}(\vT)=\xi(\vT)\cdot L_t^r(\vT)
\]
holds $P\uph\sigma(\vT)$--a.s. true.
\end{prop} 		 			

\begin{proof}
Ad (i): Note that according to Lemma \ref{lem35}, $\kappa_{\vT}(r)$ is the $P\uph\sigma(\vT)$-a.s. unique solution to the equation \eqref{lem35b}.  According to Theorem \ref{thm!}(i), there exists an
essentially unique pair	$(\ga,\xi)\in\F_{P,\ln}\times \mathcal R_+(D)$ as in (i) with condition \eqref{sch1} following by \eqref{lem35a} and standard computations. 

Ad (ii): An easy computation justifies that $\ga\in\F_{P,\ln}$; hence we may apply Theorem \ref{thm!}(ii), in order to obtain a unique probability measure   $Q:=Q^r\in\M_{S,{\bf\Lambda}(\rho(\vT))}$  determined by  condition \eqref{mart} with the martingale $M^{(\ga)}(\vT)$ fulfilling condition \eqref{sch1}. 

Ad (iii): By Theorem \ref{thm!}(iii) there exist an essentially unique rcp $\{Q_{\theta}\}_{\theta\in D}:=\{Q_{\theta}^r\}_{\theta\in{D}}$ of $Q$ over $Q_\vT$  consistent with $\vT$ and a $P_\vT$-null set $L_{\ast\ast}\in\B(D)$, satisfying for any $\theta\notin L_{\ast\ast}$ conditions $Q_\theta\in{\M}_{S,{\bf \Lambda}(\ttheta)}$ and \eqref{rcp2}, while condition \eqref{rcp2} follows by \eqref{lem35a} and standard computations. 

In particular, if $P_{W_1}$ is absolutely continuous with respect to the Lebesgue measure $\lambda$ restricted to $\mf{B}([0,1])$, then for arbitrary but fixed $r\in\R_+$ and $\theta\notin{L}_{\ast\ast}$, condition ($RRM_{\theta}$) along with condition \eqref{res1} yields 
\begin{align*}
\wt{M}^{(\ga)}_t(\theta) &= e^{r\cdot(u+c(\theta)\cdot t-r^u_t(\theta))-(\kappa_\theta(r)+c(\theta)\cdot r)\cdot T_{N_t} +\ln\E_P[e^{r\cdot X_1}]}\\
&  \cdot \frac{\int^{\infty}_{t-T_{N_t}} e^{-(\kappa_\theta(r)+c(\theta)\cdot r)\cdot w} \, (P_\theta)_{W_1}(dw)}{1-{\bf{K}}(\theta)(t-T_{N_t})}\\
&= e^{-r\cdot(r^u_t(\theta)-u) +(\kappa_\theta(r)+c(\theta)\cdot r)\cdot (t-T_{N_t}) -\kappa_\theta(r)\cdot t +\ln\E_P[e^{r\cdot X_1}]}\\
&  \cdot \frac{\int^{\infty}_{t-T_{N_t}} e^{-(\kappa_\theta(r)+c(\theta)\cdot r)\cdot w} \, (P_\theta)_{W_1}(dw)}{1-{\bf{K}}(\theta)(t-T_{N_t})}\\
 		&= L^r_t(\theta);
\end{align*}
 		hence $\wt{M}_t^{(\ga)}(\theta)=L_t^r(\theta)$. The latter, together with Lemma \ref{lem35}, implies $\wt{M}_t^{(\ga)}(\vT)=L_t^r(\vT)\;\; P\uph\sigma(\vT)$-a.s.; hence taking into account condition \eqref{sch1} we obtain $M_t^{(\ga)}(\vT)=\xi(\vT)\cdot L_t^r(\vT)\;\; P\uph\sigma(\vT)$-a.s..
\end{proof}

Finally, it is worth noticing, that for any $\theta\notin{L}_{\ast\ast}$ the two martingales $\wt{M}^{(\ga)}(\theta)$ and $L^r(\theta)$ were constructed with a totally different way. The construction of $L^r(\theta)$ for $r\geq0$ and $\theta\notin{L}_{\ast\ast}$ requires the Markovisation of the reserve process $r^u(\theta)$ $(u\in\R_+)$ and then uses the theory of piecewise deterministic Markov processes (cf. e.g. \cite{scm}, Chapter 8), while the construction of $\wt{M}^{(\ga)}(\theta)$ is simple, since it is a result of a rather elementary proof consisting of standard measure theoretic arguments. Note also that the above proposition extends Lemma 8.4 of \cite{scm} to CMRPs. Applications of Proposition \ref{Schmidli} to the ruin problem for CMRPs are presented in \cite{t1}.

 \end{document}